\documentclass[preprint,1p,10pt]{elsarticle}

\biboptions{numbers,sort&compress}

\usepackage{lineno,hyperref}
\modulolinenumbers[1]

\usepackage{array}
\newcolumntype{C}[1]{>{\centering\arraybackslash}m{#1}}

\usepackage{graphics}
\usepackage{amssymb, latexsym, amsmath, epsfig}
\usepackage{graphicx}
 \usepackage{color}
 \usepackage{epstopdf}
 \usepackage{comment}
 \usepackage{soul}
 \usepackage[normalem]{ulem}
  \usepackage{float}

\usepackage{amssymb,amsmath}
\usepackage{amsfonts,amstext,amsthm,amssymb,comment}
\usepackage{epsfig,graphics,color}

\graphicspath{{figures/}}

\newtheorem{theorem}{Theorem}
\newtheorem{lemma}{Lemma}
\newtheorem{corollary}{Corollary}

\numberwithin{equation}{section}
\newtheorem{postulate}{Postulate}

\newtheorem{remark}{Remark}

\journal{arXiv}

\begin{document}

\begin{frontmatter}

\title{Dispersion-minimized mass for isogeometric analysis}
% of eigenvalue problems in structural mechanics}
%Optimally blended quadratures  
%\tnoteref{mytitlenote}}
%\tnotetext[mytitlenote]{This paper is partially supported by...}

%\author{M. Barto$\check{\text{n}}$\fnref{myfootnote}}

\author[ad,ad1]{Quanling Deng\corref{corr}}
\author[ad,ad1,ad2]{Victor Calo}

%\author[ad]{Quanling Deng}

%\author[ad1]{Michael Barto\v{n}} Michael.Barton@centrum.cz, 

\cortext[corr]{Corresponding author. \\ E-mail addresses: Quanling.Deng@curtin.edu.au; Victor.Calo@curtin.edu.au.}
%, vladimir.puzyrev@gmail.com, vmcalo@gmail.com.

%\author[ad]{Vladimir Puzyrev}

\address[ad]{Curtin Institute for Computation, Curtin University, Kent Street, Bentley, Perth, WA 6102, Australia}

\address[ad1]{Department of Applied Geology, Western Australian School of Mines, Curtin University, Kent Street, Bentley, Perth, WA 6102, Australia}

%\address[ad1]{Basque Center for Applied Mathematics, Alameda de Mazarredo 14, 48009 Bilbao, Basque Country, Spain}
\address[ad2]{Mineral Resources, Commonwealth Scientific and Industrial Research Organisation (CSIRO), Kensington, Perth, WA 6152, Australia}

\begin{abstract}
We introduce the dispersion-minimized mass for isogeometric analysis to approximate the structural vibration which we model as a second-order differential eigenvalue problem. The dispersion-minimized mass reduces the eigenvalue error significantly, from the optimum order of $2p$ to the superconvergence order of $2p+2$ for the $p$-th order isogeometric elements with maximum continuity, which in return leads to more robust of the isogeomectric analysis. We first establish the dispersion error for arbitrary polynomial order isogeometric elements. We derive the dispersion-minimized mass in one dimension by  solving a  $p$-dimensional local matrix problem for the $p$-th order approximation and then extend it to multiple dimensions on tensor-product grids. We show that the dispersion-minimized mass can also be obtained by approximating the mass matrix using optimally blended quadratures. 
We generalize the results of optimally blended quadratures from polynomial orders $p=1,\cdots, 7$ that were studied in \cite{calo2017dispersion} to arbitrary polynomial order isogeometric approximations. Various numerical examples validate the eigenvalue  and eigenfunction error estimates we derive.
\end{abstract}


\end{frontmatter}

\linenumbers

\section{Introduction} \label{sec:intr} 
Isogeometric analysis is a widely-used numerical method introduced in 2005 \cite{hughes2005isogeometric,cottrell2009isogeometric}. The motivation was to unify the finite element methods with computer-aided design tools. Under the framework of the classic Galerkin finite element methods, isogeometric analysis uses B-splines or Non-uniform rational basis splines (NURBS) instead of the Lagrange interpolation polynomials as its basis functions. These basis functions have higher continuity (smoother), which in return improves the numerical approximations of real-life problems.

%The authors in \cite{bazilevs2006isogeometric} establishes the approximation, stability, and error estimates.  

The authors in \cite{cottrell2006isogeometric} first use isogeometric analysis to study the structural vibrations and wave propagation problems. Their spectrum analysis shows that isogeometric elements significantly improve the accuracy of the spectral approximation when compared with the classical finite elements. In \cite{hughes2014finite}, the authors explore additional advantages of isogeometric analysis on the spectral approximation over finite elements.

The dispersion analysis is well studied in literature \cite{thompson1994complex,thompson1995galerkin,ihlenburg1995dispersion,ainsworth2004discrete,harari1997reducing,harari2000analytical,he2011dispersion,guddati2004modified}  and the spectral analysis for structural vibrations (eigenvalue problems) has a strong connection with the dispersion analysis for wave propagation. The authors in \cite{hughes2008duality} introduce
a duality principle which establishes a bijective map from spectral analysis to dispersion analysis. The leading order term in the spectrum (eigenvalue) error expansion is the negative of the leading order term in the dispersion error expansion. This duality allows us to utilize the dispersion analysis tools to study the isogeometric spectral approximation properties of the eigenvalue problems; see for example the recent works \cite{bartovn2017generalization,calo2017dispersion,calo2017quadrature,deng2018dispersion,puzyrev2017dispersion,puzyrev2017spectral}.

In \cite{calo2017dispersion,puzyrev2017dispersion}, the authors propose optimally blended quadrature rules to compute the isogemetric stiffness and mass matrices. These quadratures improve the spectral approximation, in particular, the convergence rates in the eigenvalue approximation with respect to the mesh size increase by two extra orders. In \cite{hughes2014finite}, the authors state the Pythagorean eigenvalue error theorem (the proof is done in \cite{strang1973analysis}).  In \cite{puzyrev2017dispersion}, these results are generalized to include the quadrature errors  from the approximations of the inner products associated with the stiffness and mass matrices and in \cite{calo2017spectral} to include the incompatibility of the discrete spaces. Comparisons with quadrature blending rules for finite elements (see for example \cite{marfurt1984accuracy,ainsworth2010optimally}) are made in \cite{puzyrev2017dispersion} and significant error-reductions are observed in isogeometric elements. Other variants of quadrature blendings are studied as well as the superconvergence in eigenvalue errors and the optimal convergence in eigenfunction errors are established in \cite{calo2017dispersion}. To reduce the computational costs for blending two quadrature rules, a single non-standard quadrature rule is developed in \cite{deng2018dispersion} for $C^1$ quadratic isogeometric elements. The paper \cite{puzyrev2017spectral} studies the stopping bands and outliers in the numerical spectral approximations of the classic finite elements or isogeometric elements with variable continuities.

The mass-lumping technique and all the above works (also \cite{harari1997reducing,guddati2004modified,he2011dispersion,yue2005dispersion,ewing1983incorporation}) take the advantage of the mass matrix to reduce the dispersion or spectrum errors.
In this paper, we generalize this collection of insights to introduce the idea of a dispersion-minimized mass. We first establish the dispersion error, which is of the optimal order $2p$ where $p$ is the polynomial approximation order, for isogeometric elements with B-splines of maximum continuity. We view the entries in the mass matrix as  degrees of freedom which allows us to optimize the dispersion errors to be of order $2p+2$. The dispersion error is thus minimized and we refer to these corresponding mass entries as the dispersion-minimized mass entries. To find the dispersion-minimized mass entries, we propose a $p$-dimensional local linear system. The system is non-singular for arbitrary $p$ and it's computationally stable and efficient to invert as the dimension is low.

We also minimize the dispersion error for isogeometric analysis by blending quadrature rules optimally. The optimal blending parameters are given for arbitrary order $p$, which is a generalization of the work \cite{calo2017dispersion}. The generalization is not only on $p$ (from $p=1,\cdots, 7$ to arbitrary) but also on other quadrature rules, that is, not limited to the blendings of the Gauss-Legendre and Gauss-Lobatto rules as in \cite{calo2017dispersion}. These optimally blended rules also lead to a superconvergence of order $2p+2$. In fact, we show that the dispersion-minimized mass entries are the same as those obtained by optimally blended rules. We establish our findings in one dimension and then extend them to multiple dimensions by using the tensor-product grids (see also \cite{ainsworth2010optimally,calo2017dispersion}).

The rest of this paper is organized as follows. Section \ref{sec:ps} presents the problem and its discretization as well as introduces some relevant quadrature rules. In Section \ref{sec:disperr}, we develop several new facts for finite elements with B-spline basis functions, then we present the discrete dispersion errors for arbitrary order isogeometric elements. In Section \ref{sec:dmm}, we introduce the idea of dispersion-minimized mass and establish a superconvergence result of order $2p+2$ for the eigenvalue errors. Section \ref{sec:ob} generalizes the optimal blending quadrature rules of \cite{calo2017dispersion}. We state our main theoretical results in Sections \ref{sec:disperr}
to %, \ref{sec:dmm}, and 
\ref{sec:ob}.
Following the work \cite{calo2017dispersion}, Section \ref{sec:md} presents the generalization to multiple dimensions and the eigenfunction error estimates, while Section \ref{sec:num} collects numerical examples that demonstrate the performance of the proposed blending schemes. 
Concluding remarks are presented in Section \ref{sec:conclusion}.

%To the best of our knowledge, this is the first paper studying the mixed isogeometric analysis of the biharmonic eigenvalue problem with error analysis on the eigenvalue and eigenfunctions.

% the design of optimal quadrature rules which  minimize the dispersion errors of the isogeometric analysis for the wave propagation and structural vibration problems. The dispersion error-minimizing quadratures, that combine Gauss-Legendre  

%The rest of this paper is organized as follows. Section \ref{sec:ps} describes the isogeometric discretization of an eigenvalue problem. In Section \ref{sec:ddr}, we present the constraints minimizing both dispersion error and the number of quadrature points and set up the equations for the quadrature weights and points. Both two-point and 2.5-point rules are considered here. Section \ref{sec:num} studies numerical examples to demonstrate the performance of both the two-point and 2.5-point rules. Concluding remarks are given in Section \ref{sec:conclusion}.

\section{Problem setting} \label{sec:ps}

The classical second-order differential eigenvalue problem that arises in structural mechanics is to find the vibration frequencies $\omega$ and vibration modes $u$ such that
\begin{equation} \label{eq:pde}
\begin{aligned}
- \Delta u & =  \lambda u \quad  \text{in} \quad \Omega, \\
u & = 0 \quad \text{on} \quad \partial \Omega,
\end{aligned}
\end{equation}
where $\lambda = \omega^2$, $\Delta = \nabla^2$ is the Laplacian, $\Omega \subset  \mathbb{R}^d, d=1,2,3$, is a bounded open domain with Lipschitz boundary.
The eigenvalue problem \eqref{eq:pde} has a countable set of eigenvalues $\lambda_j \in \mathbb{R}^+$ \cite{strang1973analysis}
\begin{equation}
0 < \lambda_1 < \lambda_2 \leq \lambda_3 \leq \cdots
\end{equation}
and an associated set of orthonormal eigenfunctions $u_j$, that is
\begin{equation} \label{eq:onu}
(u_j, u_k) = \delta_{jk}, 
\end{equation}
where $(\cdot, \cdot)$ denotes the $L^2-$inner product on $\Omega$ and $\delta_{lm} =1$ when $l=m$ and zero otherwise and is known as the Kronecker delta.

\subsection{Isogeometric discretization}
To discretize \eqref{eq:pde} with isogeometric elements, we first assume that $ \Omega$ is a cube and a uniform tensor product mesh of size $h_x>0, h_y>0, h_z>0$ is placed on $\Omega$. We denote each element as $K$ and its collection as $\mathcal{T}_h$ such that $\bar\Omega = \cup_{K\in \mathcal{T}_h}  K$. Let $h = \max_{K\in \mathcal{T}_h} \text{diameter}(K)$. 
The variational formulation of \eqref{eq:pde} at the continuous level is to find $\lambda \in \mathbb{R}^{+}$ and $u \in H^1_0(\Omega)$ such that 
\begin{equation} \label{eq:vf}
a(w, u) =  \lambda b(w, u), \quad \forall \ w \in H^1_0(\Omega), 
\end{equation}
where 
$
a(w, v) = (\nabla w, \nabla v)
$ and $
b(w, v) = (w, v)
$. 
Here, we denote by $H^m(\Omega)$ the Sobolev-Hilbert spaces and $H^m_0(\Omega)$ the Sobolev-Hilbert spaces with functions vanishing at the boundary for $m>0$, where $m$ specifies the order of weak derivatives. From \eqref{eq:onu}, the normalized eigenfunctions are also orthogonal in the energy inner product
\begin{equation} \label{eq:vfo}
a(u_j, u_k) =  \lambda_j b(u_j, u_k) = \lambda_j \delta_{jk}.
\end{equation}

By specifying a finite dimensional approximation space $V_h \subset H^1_0(\Omega)$ where $V_h = \text{span} \{\phi_a\}$ is the span of the B-spline basis functions $\phi_a$, the isogeometric analysis of \eqref{eq:pde} seeks $\lambda^h \in \mathbb{R}$ and $u^h \in V_h$ such that 
\begin{equation} \label{eq:vfh}
a(w^h, u^h) =  \lambda^h b(w^h, u^h), \quad \forall \ w^h \in V_h.
\end{equation}

The definition of the B-spline basis functions in one dimension is as follows. 
Let $X = \{x_0, x_1, \cdots, x_m \}$ be a knot vector with knots $x_j$, that is, a nondecreasing sequence of real numbers which are called knots.  The $j$-th B-spline basis function of degree $p$, denoted as $\theta^j_p(x)$, is defined as \cite{de1978practical,piegl2012nurbs}
\begin{equation} \label{eq:B-spline}
\begin{aligned}
\theta^j_0(x) & = 
\begin{cases}
1, \quad \text{if} \ x_j \le x < x_{j+1} \\
0, \quad \text{otherwise} \\
\end{cases} \\ 
\theta^j_p(x) & = \frac{x - x_j}{x_{j+p} - x_j} \theta^j_{p-1}(x) + \frac{x_{j+p+1} - x}{x_{j+p+1} - x_{j+1}} \theta^{j+1}_{p-1}(x).
\end{aligned}
\end{equation}

In this paper, we utilize the B-splines on uniform meshes with non-repeating knots, that is, we use B-splines with maximum continuity on uniform meshes.
We approximate the eigenfunctions as a linear combination of the B-spline basis functions and substitute all the B-spline basis functions for $v_h$ in \eqref{eq:vfh} which leads to the matrix eigenvalue problem
\begin{equation} \label{eq:mevp}
\mathbf{K} \mathbf{U} = \lambda^h \mathbf{M} \mathbf{U},
\end{equation}
where $\mathbf{K}_{ab} =  a(\phi_a, \phi_b), \mathbf{M}_{ab} = b(\phi_a, \phi_b),$ and $\mathbf{U}$ is the corresponding representation of the eigenvector as the coefficients of the B-spline basis functions.

\subsection{Quadrature rules}
In practice, we evaluate the integrals involved in $a(u_j^h, v_h) $ and $b(u_j^h, v_h)$ numerically, that is, approximated by quadrature rules. On a reference element $\hat K$, a quadrature rule is of the form
\begin{equation} \label{eq:qr}
\int_{\hat K} \hat f(\hat{\boldsymbol{x}}) \ \text{d} \hat{\boldsymbol{x}} \approx \sum_{l=1}^{N_q} \hat{\varpi}_l \hat f (\hat{n_l}),
\end{equation}
where $\hat{\varpi}_l$ are the weights, $\hat{n_l}$ are the nodes, and $N_q$ is the number of quadrature points. For each element $K$, we assume that there is an invertible map $\sigma$ such that $K = \sigma(\hat K)$, which leads to the correspondence between the functions on $K$ and $\hat K$. Assuming $J_K$ is the corresponding Jacobian of the mapping, \eqref{eq:qr} induces a quadrature rule over the element $K$ given by
\begin{equation} \label{eq:q}
\int_{K}  f(\boldsymbol{x}) \ \text{d} \boldsymbol{x} \approx \sum_{l=1}^{N_q} \varpi_{l,K} f (n_{l,K}),
\end{equation}
where $\varpi_{l,K} = \text{det}(J_K) \hat \varpi_l$ and $n_{l,K} = \sigma(\hat n_l)$. 
For simplicity, we denote by $G_m$ the $m-$point Gauss-Legendre quadrature rule,  by $L_m$ the $m-$point Gauss-Lobatto quadrature rule, by $R_m$ the $m-$point Gauss-Radau quadrature rule, and by $O_p$ the optimal blending scheme for the $p$-th order isogeometric analysis with maximum continuity. In one dimension, $G_m, L_m$, and $R_m$ fully integrates polynomials of order $2m-1, 2m-3,$ and $2m-2$, respectively \cite{bartovn2016optimal,bartovn2016gaussian,bartovn2017gauss}.  

Applying quadrature rules to \eqref{eq:vfh}, we have the approximated form
\begin{equation} \label{eq:vfho}
 a_h(w^h, \tilde u^h) =  \tilde\lambda^h  b_h(w^h, \tilde u^h), \quad \forall \ w^h \in V_h,
\end{equation}
where for $w,v \in V_h$
\begin{equation} \label{eq:ba}
 a_h(w, v) = \sum_{K \in \mathcal{T}_h} \sum_{l=1}^{N_q} \varpi_{l,K}^{(1)} \nabla w (n_{l,K}^{(1)} ) \cdot \nabla v (n_{l,K}^{(1)} )
\end{equation}
and
\begin{equation} \label{eq:bb}
 b_h(w, v) = \sum_{K \in \mathcal{T}_h} \sum_{l=1}^{N_q} \varpi_{l,K}^{(2)} w (n_{l,K}^{(2)} ) v (n_{l,K}^{(2)} )
\end{equation}
where $\{\varpi_{l,K}^{(1)}, n_{l,K}^{(1)} \}$ and $\{\varpi_{l,K}^{(2)}, n_{l,K}^{(2)} \}$ specify two (possibly different) quadrature rules. In one dimension for the $p$-th order isogeometric elements, $G_{p+1}$ integrates these two bilinear forms exactly, that is, for $w,v \in V_h$
\begin{equation} \label{eq:ab=abh}
 a_h(w, v) =  a(w, v), \qquad  b_h(w, v)  =  b(w, v),
\end{equation}
while both $G_p$ and $L_{p+1}$ integrate $a(w, v)$ exactly but under-integrate $b(w,v)$.
With quadrature rules, we can rewrite the matrix eigenvalue problem \eqref{eq:mevp} as
\begin{equation} \label{eq:amevp}
\mathbf{K} \tilde{\mathbf{U}} = \tilde \lambda^h \mathbf{M} \tilde{\mathbf{U}}
\end{equation}
where $\mathbf{K}_{ab} =  a_h(\phi_a, \phi_b), \mathbf{M}_{ab} = b_h(\phi_a, \phi_b),$ and $\tilde{\mathbf{U}}$ is the corresponding representation of the eigenvector as the coefficients of the basis functions.

\section{Dispersion error in 1D} \label{sec:disperr}
In the view of duality principle \cite{hughes2008duality}, which establishes a unified analysis between the spectral analysis for eigenvalue problems and the dispersion analysis for wave propagations, we establish the eigenvalue error estimates by studying the dispersion errors of isogeometric elements for \eqref{eq:pde} with a generic eigen-frequency.  The dispersion and spectrum analysis are also unified in the form of a Taylor expansion for eigenvalue errors in \cite{calo2017dispersion}.

Now, we study the dispersion analysis of the isogeometric elements for \eqref{eq:pde}. The dispersion analysis studies the numerical approximation of the well-known Helmholtz equation 
\begin{equation} \label{eq:pdee}
\begin{aligned}
- \Delta u -  \omega^2 u & = 0 \quad  \text{in} \quad \Omega, \\
u & = 0 \quad \text{on} \quad \partial \Omega,
\end{aligned}
\end{equation}
which we discretize in the same fashion as for the eigenvalue problems, that is
\begin{equation} \label{eq:Hvf}
\begin{aligned}
a(w, u) -  \omega^2 b(w, u) & = 0, \quad \forall \ w \in H^1_0(\Omega), \\
a(w^h, u^h) -  \omega^2 b(w^h, u^h) & = 0, \quad \forall \ w^h \in V_h, \\
a_h(w^h, u^h) -  \omega^2 b_h(w^h, u^h) & = 0, \quad \forall \ w^h \in V_h.
\end{aligned}
\end{equation}

Suppose we utilize the $p$-th order B-spline basis functions in the bilinear forms \eqref{eq:ba} and \eqref{eq:bb} on a uniform mesh of size $h>0$ in 1D and seek an approximation of the eigenfunction the form 
\begin{equation}
\sum_{j} U^j_p \theta^j_p(x),
\end{equation}
where $U^j_p$ are the unknown coefficients which corresponds to the the $p$-th order polynomial approximation which are to be determined.

The classical dispersion analysis of wave propagation problems relies on the Bloch wave assumption \cite{odeh1964partial}, which states that \eqref{eq:pde} admits nontrivial Bloch wave solutions in the form 
\begin{equation} \label{eq:bloch}
U^j_p = e^{ij \mu h },
\end{equation}
 where $i^2 = -1$ and $\mu$ is an approximated frequency. The $C^{p-1}$ B-spline basis function $\theta^j_p$ has a support over $p+1$ elements. Thus, we have
\begin{equation}
\begin{aligned}
a(\theta^j_p, u^h) & = a \Big(\theta^j_p, \sum_{|k-j| \le p} U^k_p \theta^k_p \Big) = A_pU_p / h, \\
b(\theta^j_p, u^h) & = b \Big(\theta^j_p, \sum_{|k-j| \le p} U^k_p \theta^k_p \Big) = B_pU_p h, \\
\end{aligned}
\end{equation}
where 
\begin{equation} \label{eq:UAB}
\begin{aligned}
U_p & = [U^{j-p}_p \quad U^{j-p+1}_p \quad \cdots \quad U^{j}_p \quad \cdots \quad U^{j+p-1}_p \quad U^{j+p}_p ]^T, \\
A_p & = [A^{j-p}_p \quad A^{j-p+1}_p \quad \cdots \quad A^{j}_p \quad \cdots \quad A^{j+p-1}_p \quad A^{j+p}_p ], \\
B_p & = [B^{j-p}_p \quad B^{j-p+1}_p \quad \cdots \quad B^{j}_p \quad \cdots \quad B^{j+p-1}_p \quad B^{j+p}_p ], \\
\end{aligned}
\end{equation}
with 
\begin{equation} \label{eq:ABdefinition}
A^{j-k}_p = a(\theta^j_p, \theta^{j-k}_p) h, \qquad B^{j-k}_p = b(\theta^j_p, \theta^{j-k}_p) / h
\end{equation}
 for $k=p, p-1, \cdots, -p$. The symmetry of the B-spline basis functions  (on uniform meshes and away from the boundaries) further implies that 
\begin{equation} \label{eq:symm}
A^{j-k}_p = A^{j+k}_p, \qquad B^{j-k}_p = B^{j+k}_p.
\end{equation}
Also, the local support of $\theta^j_p$ implies 
\begin{equation} \label{eq:0}
A^{j-k}_p = B^{j-k}_p = 0, \quad \forall \ k > p  \ \text{or }  k < -p.
\end{equation}

%[a(\theta_{j-p}, \theta^j_p) \quad a(\theta_{j-p+1}, \theta^j_p)  \quad \cdots \quad  a(\theta_{j}, \theta^j_p)  \quad \cdots \quad a(\theta_{j+p-1}, \theta^j_p)  \quad a(\theta_{j+p}, \theta^j_p)  ], \\
%[b(\theta_{j-p}, \theta^j_p) \quad b(\theta_{j-p+1}, \theta^j_p)  \quad \cdots \quad  b(\theta_{j}, \theta^j_p)  \quad \cdots \quad b(\theta_{j+p-1}, \theta^j_p)  \quad b(\theta_{j+p}, \theta^j_p)  ], \\

Thus, using the symmetry of \eqref{eq:symm}, the Bloch wave assumption \eqref{eq:bloch} and Euler's formula, one can calculate
\begin{equation} \label{eq:abeh}
\begin{aligned}
a(\theta^j_p, u^h) & = A_pU_p / h = \big(A^j_p + 2\sum_{k=1}^p A^{j+k}_p \cos(k\mu h) \big) e^{ij \mu h } / h, \\
b(\theta^j_p, u^h) & = B_pU_p h = \big(B^j_p + 2\sum_{k=1}^p B^{j+k}_p \cos(k\mu h) \big) e^{ij \mu h } h. \\
\end{aligned}
\end{equation}

Before we derive the dispersion error, we first establish a few lemmas for any order B-spline basis functions with maximum continuity, that is, $C^{p-1}$ for the $p$-th order B-spline basis functions  on a uniform grid on the real line.

\subsection{Preliminary results on B-splines} \label{sec:bspline}
Firstly, we list several known results on the stiffness and mass matrices. In this subsection, we assume that both stiffness and mass matrix entries are integrated exactly and the B-splines of degree $p$ are $C^{p-1}$ and defined on a uniform grid on the one dimensional real number line. 

\begin{lemma} \label{lem:bsp}
The B-splines are symmetric, that is, 
\begin{equation}
\theta^j_p (x) = \theta^j_p (j+ p+1 - x),
\end{equation}
and strictly monotone on $[x_j, x_{j+(p+1)/2}]$ and $[x_{j+(p+1)/2}, x_{j+p+1}]$. Moreover, the scalar products of the B-splines $\theta^j_p (x)$ and $\theta^{j+k}_p (x)$ and of their derivatives satisfy 
\begin{equation} \label{eq:smrec}
\begin{aligned}
A^{j+k}_p &  = 2 B^{j+k}_{p-1} - B^{j+k+1}_{p-1}  - B^{j+k-1}_{p-1}, \\
B^{j+k}_p &  = \theta^j_{2p+1} (j+ k + p+1) = \theta^j_{2p+1} (j- k + p+1), \\
\end{aligned}
\end{equation}
respectively. Lastly, 
\begin{equation} \label{eq:pu}
\begin{aligned}
\sum_{k=-p}^p A^{j+k}_p = \sum_{k=-p}^p B^{j+k}_p  - 1 = 0.
\end{aligned}
\end{equation}
\end{lemma}

\begin{proof}
Symmetry and monotonicity are obvious. The scalar product properties in \eqref{eq:smrec} are direct results from \cite{hollig2003finite} by replacing its B-spline basis functions appropriately. The summation properties \eqref{eq:pu} are true as the B-splines satisfy the partition of unity.
\end{proof}

Invoking the definition of B-splines \eqref{eq:B-spline}, Lemma \ref{lem:bsp} implies the following recursive formula for the mass entries
\begin{equation} \label{eq:Bk2Bk-1}
\begin{aligned}
B^{j+k}_p & = \frac{(p+k+1)^2 B^{j+k+1}_{p-1} - 2(k^2 - p -p^2)B^{j+k}_{p-1} + (p-k+1)^2 B^{j+k-1}_{p-1} }{2p (2p+1) }.
%A^{j-k}_p & =  2  \theta^j_{2p-1} (j - k + p) -  \theta^j_{2p-1} (j - k - 1+ p) -  \theta^j_{2p-1} (j - k + 1 +  p), \\
%A^{j+k}_p & =  2  \theta^j_{2p-1} (j + k + p) -  \theta^j_{2p-1} (j + k - 1+ p) -  \theta^j_{2p-1} (j + k + 1 +  p). \\
\end{aligned}
\end{equation}

\begin{lemma} \label{lem:Ak}
For any positive integer $p$, there holds
\begin{equation} \label{eq:euler}
\begin{aligned}
\Big( \sum_{k=1}^p A^{j+k}_p k^2 \Big) + 1 =  0.
\end{aligned}
\end{equation}
\end{lemma}

\begin{proof}
Applying \eqref{eq:symm}, \eqref{eq:0}, and the first equality of \eqref{eq:smrec}, we obtain
\begin{equation} 
\begin{aligned}
\sum_{k=1}^p A^{j+k}_p k^2 & = \sum_{k=1}^p k^2 \big( 2 B^{j+k}_{p-1} - B^{j+k+1}_{p-1}  - B^{j+k-1}_{p-1} \big) \\
& = - B_{p-1}^j - 2 B_{p-1}^{j+1} + \sum_{k=2}^{p-1} \Big( -(k-1)^2 + 2k^2 - (k+1)^2 \Big) B_{p-1}^{j+k} \\
& \quad +\Big(-(p-1)^2 + 2p^2 \Big) B_{p-1}^{j+p} - p^2 B_{p-1}^{j+p+1} \\
& = - B_{p-1}^j  - 2 \sum_{k=1}^{p-1} B_{p-1}^{j+k} \\
& = - \sum_{k=1- p}^{p-1} B^{j+k}_{p-1}.
\end{aligned}
\end{equation}

Applying \eqref{eq:pu} with $p-1$, we arrive to
\begin{equation} 
\begin{aligned}
\sum_{k=1}^p A^{j+k}_p k^2 + 1 = - \sum_{k=1- p}^{p-1} B^{j+k}_{p-1} + 1 = 0,
\end{aligned}
\end{equation}
which completes the proof.
\end{proof}

%We point out that the term $A_p^{j+k}$ is in terms of $(2p+1)!$ and for $p\to \infty$ this identity shares the feature of the Euler's identity in real number field. 

\begin{lemma} \label{lem:Bk}
For any positive integer $p$, there holds
\begin{equation} 
\begin{aligned}
p+1 - 12 \sum_{k=1}^p B^{j+k}_p k^2  =  0.
\end{aligned}
\end{equation}
\end{lemma}

\begin{proof}
We prove this by induction on $p$. Firstly, for $p=1$, we have 
\begin{equation*} 
\begin{aligned}
p+1 - 12 \sum_{k=1}^p B^{j+k}_p k^2  =  1+1 - 12 \times 1^2 \times \frac{1}{6} = 0.
\end{aligned}
\end{equation*}
Now, suppose it is true for $p=q-1$. Then for $p=q$, invoking \eqref{eq:symm}, \eqref{eq:0},  and \eqref{eq:Bk2Bk-1} gives
\begin{equation*} 
\begin{aligned}
p+1 - 12 \sum_{k=1}^p B^{j+k}_p k^2  & =  q +1 - 12 \sum_{k=1}^q B^{j+k}_q k^2 \\
& = q +1 - \frac{12}{2q(2q+1)}  \Big( q^2 + (2-6q+4q^2) \sum_{k=1}^{q-1} k^2 B^{j+k}_{q-1}  \Big) \\
& = q +1 - \frac{12}{2q(2q+1)}  \Big( q^2 + (2-6q+4q^2) \frac{q}{12}  \Big) \\
& = 0,
\end{aligned}
\end{equation*}
which completes the proof.
\end{proof}

In order to proceed with the next Lemma, we define
\begin{equation} \label{eq:defineFG}
\begin{aligned}
F^0_{p,m} & = -2p(2p+1) + 2m (2m-1) p^2, \\
G^0_{p,m,k} & = 2p(2p+1) \Big( 2k^{2m} - (k+1)^{2m} - (k-1)^{2m} \Big) \\
& \quad + 2m (2m-1)  \Big( (k-1)^{2m-2} (p+k)^2 - 2k^{2m-2} (k^2- p - p^2) \\
& \qquad \qquad \qquad \qquad  + (k+1)^{2m-2} (p-k)^2 \Big), \\
% \frac{1}{(2p-2q) (2p-2q+1)}
\text{and for} \ q &= 1,2, \cdots, p-2, \\
F^q_{p,m} & =  2(p-q+1)(p-q) F^{q-1}_{p,m} + (p-q)^2G^{q-1}_{p,m,k}, \\ %\qquad q = 1,2, \cdots, p-2, \\
G^{q}_{p,m,k} & = (p-q+k)^2 G^{q-1}_{p,m,k-1} - 2 \big( k^2 - (p-q+1)(p-q) \big) G^{q-1}_{p,m,k} \\
& \quad + (p-q-k)^2 G^{q-1}_{p,m,k+1}. \\ %\qquad q = 1,2, \cdots, p-2.\\
\end{aligned}
\end{equation}

Now, we postulate the following on these terms.
\begin{postulate} \label{lem:FkGk}
For any positive integer $p>1$ and $m=2,3,\cdots, p$, there holds
\begin{equation} 
\begin{aligned}
2 F^{q}_{p+1,m} - G^{q}_{p+1,m,0} & = 0, \qquad q = 1,2, \cdots, p-2, \\
4 F^{p-2}_{p,m} + G^{p-2}_{p,m,1} & = 0.
\end{aligned}
\end{equation}
\end{postulate}

\begin{proof}
These two identities are in terms of integers. These statements are verified for various numbers using  \textbf{Mathematica} \cite{mathematica}. In our case, we verified these statements up to the largest numbers $p=m=17.$ The first identity can be generalized for any $q$.
\end{proof}

\begin{lemma} \label{lem:AkBk}
For any positive integer $p>1$ and $m=2,3,\cdots, p$, there holds
\begin{equation} \label{eq:AkBk}
\sum_{k=1}^p \Big( \frac{k^{2m}}{(2m)!} A_p^{j+k}  + \frac{k^{2m-2}}{(2m-2)!} B_p^{j+k} \Big) = 0.
\end{equation}
\end{lemma}

\begin{proof}
We prove this by induction on $p$ and $m$. Denote the left-hand-side term in \eqref{eq:AkBk} as $T_{p,m}$ such that 
\begin{equation*} 
\begin{aligned}
T_{p,m}  & =  T_{p,m}^A + T_{p,m-1}^B,
\end{aligned}
\end{equation*}
where
\begin{equation*} 
\begin{aligned}
T_{p,q}^A = \sum_{k=1}^p \frac{k^{2q}}{(2q)!} A_p^{j+k}, \qquad T_{p,q}^B = \sum_{k=1}^p \frac{k^{2q}}{(2q)!} B_p^{j+k}.
\end{aligned}
\end{equation*}
Then using \eqref{eq:0},  \eqref{eq:smrec}, and \eqref{eq:Bk2Bk-1} gives 
\begin{equation*} 
\begin{aligned}
T_{p,m}^A  & =  \frac{1}{(2m)!} \Big( -B_{p-1}^j + \sum_{k=1}^{p-1} \big( -(k-1)^{2m} + 2k^{2m} - (k+1)^{2m} \big) B_{p-1}^{j+k} \Big), \\
T_{p,m}^B  & =  \frac{1}{(2m)!(2p)(2p+1)} \Big( p^2 B_{p-1}^j \\
& \quad + \sum_{k=1}^{p-1} \big( (k-1)^{2m} (p+k)^2 - 2k^{2m}(k^2-p-p^2) + (k+1)^{2m} (p-k)^2 \big) B_{p-1}^{j+k} \Big).
\end{aligned}
\end{equation*}
Therefore, using the notation of \eqref{eq:defineFG} implies
\begin{equation*} 
\begin{aligned}
T_{p,m}  & =  T_{p,m}^A + T_{p,m-1}^B = \frac{1}{(2m)!(2p)(2p+1)} \Big( F^0_{p,m} B_{p-1}^j + \sum_{k=1}^{p-1} G^0_{p,m,k} B_{p-1}^{j+k}  \Big).
\end{aligned}
\end{equation*}

Using Postulate \ref{lem:FkGk} and applying \eqref{eq:symm}, \eqref{eq:0},  \eqref{eq:smrec}, and \eqref{eq:Bk2Bk-1} recursively, we obtain  
\begin{equation*} 
\begin{aligned}
T_{p,m}  & =  \frac{1}{(2m)!(2p)(2p+1)} \Big( F^0_{p,m} B_{p-1}^j + \sum_{k=1}^{p-1} G^0_{p,m,k} B_{p-1}^{j+k}  \Big) \\
& =  \frac{1}{(2m)!(2p)(2p+1)(2p-2)(2p-1)} \Big( F^1_{p,m} B_{p-1}^j + \sum_{k=1}^{p-2} G^1_{p,m,k} B_{p-1}^{j+k}  \Big) \\
& = \cdots \\
& = \frac{3!}{(2m)!(2p)!} \big( F^{p-2}_{p,m} B_1^j + G^{p-2}_{p,m,1} B_1^{j+1}  \big) \\
& = \frac{3!}{(2m)!(2p)!} \big( F^{p-2}_{p,m} \cdot \frac{2}{3} + G^{p-2}_{p,m,1} \cdot \frac{1}{6} \big) \\
& = \frac{1}{(2m)!(2p)!} \big( 4 F^{p-2}_{p,m} + G^{p-2}_{p,m,1} \big) \\
& = 0,
\end{aligned}
\end{equation*}
which completes the proof.
\end{proof}

\begin{remark}
For the particular case $m=2$, invoking \eqref{eq:pu} and Lemma \ref{lem:Bk} yields
\begin{equation*} 
\begin{aligned}
T_{p,2}  & =   \frac{2(-1+ 4p)}{4!(2p)(2p+1)}   \Big( p B_{p-1}^j - 2 \sum_{k=1}^{p-1} \big( 6k^2-p \big) B_{p-1}^{j+k} \Big) \\
& =   \frac{2(-1+ 4p)}{4!(2p)(2p+1)}   \Big( p - 12 \sum_{k=1}^{p-1} k^2 B_{p-1}^k \Big) \\
& = 0.
\end{aligned}
\end{equation*}
\end{remark}

\begin{lemma} \label{lem:C2m}
Let $C_2 = 1$ and for $m=2, 3, \cdots$, define
\begin{equation} \label{eq:c2m}
C_{2m} = \sum_{k=1}^p \frac{(-1)^m k^{2m}}{(2m)!} A_p^{j+k}  - \sum_{q=1}^{m-1}  \sum_{k=1}^p C_{2m-2q}\frac{(-1)^q k^{2q}}{(2q)!} B_p^{j+k}.
\end{equation}

For any positive integer $p\ge 2$ and $m=2,3,\cdots, p$, there holds
\begin{equation}
\begin{aligned}
C_{2m} & = 0.
\end{aligned}
\end{equation}

\end{lemma}

\begin{proof}
We prove this by induction on $m$ for $m=2,3, \cdots, p$. Firstly, for $m=2$, using $C_2=1$ and Lemma \ref{lem:AkBk}, \eqref{eq:c2m} reduces to 
\begin{equation}
C_{4} = \sum_{k=1}^p \Big( \frac{k^{4}}{4!} A_p^{j+k}  + \frac{k^{2}}{2!} B_p^{j+k} \Big) = 0.
\end{equation}
Now, assume that $C_{2m} = 0$, for $m=2,3,\cdots, s,$ where $s<p$. Then using $C_2=1$, \eqref{eq:c2m} with $m=s+1$ reduces to 
\begin{equation}
C_{2s+2} = \sum_{k=1}^p \frac{(-1)^{s+1} k^{2s+2}}{(2s+2)!} A_p^{j+k}  - \sum_{k=1}^p \frac{(-1)^s k^{2s}}{(2s)!} B_p^{j+k}.
\end{equation}
By Lemma \ref{lem:AkBk}, $C_{2s+2} = 0$ for $s=2,3,\cdots, p-1$. This completes the proof.
\end{proof}

\begin{lemma} \label{lem:aoverb}
Denote $\Lambda = \mu h.$ For any positive integer $p$, there holds
\begin{equation}
\begin{aligned}
\frac{A_p U_p }{B_p U_p} = \Lambda^2 + 2 (-1)^{p+1} \Big( \sum_{k=1}^p \frac{k^{2p+2}}{(2p+2)!} A^{j+k}_p  + \frac{k^{2p}}{(2p)!} B^{j+k}_p  \Big) \Lambda^{2p+2} + \mathcal{O}(\Lambda^{2p+4}).
\end{aligned}
\end{equation}
\end{lemma}

\begin{proof}
Assume that 
\begin{equation} \label{eq:abexpansion}
\begin{aligned}
\frac{A_p U_p }{B_p U_p} =  c_0 + c_1\Lambda + c_2 \Lambda^2 + \cdots,
\end{aligned}
\end{equation}

Applying \eqref{eq:abeh} gives
\begin{equation}
\begin{aligned}
\frac{A_p U_p }{B_p U_p} = \frac{A^j_p + 2\sum_{k=1}^p A^{j+k}_p \cos(k\Lambda) }{B^j_p + 2\sum_{k=1}^p B^{j+k}_p \cos(k\Lambda)} = c_0 + c_1\Lambda + c_2 \Lambda^2 + \cdots,
\end{aligned}
\end{equation}
which we express as
\begin{equation}
\begin{aligned}
A^j_p + 2\sum_{k=1}^p A^{j+k}_p \cos(k\Lambda) = \Big( B^j_p + 2\sum_{k=1}^p B^{j+k}_p \cos(k\Lambda) \Big) ( c_0 + c_1\Lambda + c_2 \Lambda^2 + \cdots ).
\end{aligned}
\end{equation}
Expanding $\cos(k \Lambda)$ around $\Lambda = 0$, we obtain 
\begin{equation}
\begin{aligned}
 \cos(k\Lambda) = \sum_{m=0}^\infty \frac{(-1)^m}{(2m)!} (k \Lambda)^{2m} = 1 - \frac{(k\Lambda)^2}{2!} + \cdots,
\end{aligned}
\end{equation}
and thus,
\begin{equation}
\begin{aligned}
A^j_p + 2\sum_{k=1}^p A^{j+k}_p \Big( \sum_{m=0}^\infty \frac{(-1)^m}{(2m)!} (k \Lambda)^{2m} \Big) = & \Big( B^j_p + 2\sum_{k=1}^p B^{j+k}_p \sum_{m=0}^\infty \frac{(-1)^m}{(2m)!} (k \Lambda)^{2m}  \Big) \\
& \cdot ( c_0 + c_1\Lambda + c_2 \Lambda^2 + \cdots ).
\end{aligned}
\end{equation}

Setting up equalities on the coefficients of the terms with the same powers of $k\Lambda$ and using the expression of symmetry \eqref{eq:symm}, one obtains
\begin{equation} \label{eq:cs}
\begin{aligned}
c_{2q+1} & = 0, \forall \ q=0, 1,2,\cdots, \\
\sum_{k=-p}^p A^{j+k}_p & = c_0 \sum_{k=-p}^p B^{j+k}_p, \\
2 \sum_{k=1}^p -\frac{k^2}{2!} A^{j+k}_p & = 2c_0 \sum_{k=1}^p -\frac{k^2}{2!} B^{j+k}_p + c_2 \sum_{k=-p}^p B^{j+k}_p, \\
2 \sum_{k=1}^p  \frac{(-1)^m k^{2m}}{(2m)!} A^{j+k}_p & = c_{2m} \sum_{k=-p}^p B^{j+k}_p + 2 \sum_{q=1}^{m}  \sum_{k=1}^p c_{2m-2q} \frac{(-1)^q k^{2q}}{(2q)!} B_p^{j+k}, \\
\end{aligned}
\end{equation}
where $m=2, 3, \cdots.$ Using \eqref{eq:pu} and Lemma \ref{lem:Ak} yields $c_0 = 0$ and  $c_2 = 1,$ respectively.
By a factor of 2, Lemma \ref{lem:C2m} immediately implies that $c_{2m} = 0$ for $m=2, 3,\cdots, p.$ Setting $m=p+1$ in the last equation in \eqref{eq:cs}, one obtains
\begin{equation}
\begin{aligned}
c_{2p+2} = 2 \sum_{k=1}^p  \frac{(-1)^{p+1} k^{2p+2}}{(2p+2)!} A^{j+k}_p - 2 \sum_{k=1}^p  \frac{(-1)^{p} k^{2p}}{(2p)!} B^{j+k}_p,
\end{aligned}
\end{equation}
which is substituted back to \eqref{eq:abexpansion} to complete the proof.
\end{proof}

\subsection{Dispersion error equation}
In this subsection, we assume that both the stiffness and the mass matrix entries are integrated exactly and the B-splines of degree $p$ are $C^{p-1}$ and defined on a uniform grid with $0<h<1.$ Now we present the main theorem.

\begin{theorem} \label{thm:eigenexpansion}
For each discrete mode $\omega_h$, there holds the discrete dispersion error
\begin{equation} \label{eq:disperr}
\omega^2_h - \mu^2 = 2 (-1)^{p+1} \Big( \sum_{k=1}^p \frac{k^{2p+2}}{(2p+2)!} A^{j+k}_p  + \frac{k^{2p}}{(2p)!} B^{j+k}_p  \Big) \mu^{2p+2} h^{2p} + \mathcal{O}(h^{2p+2}).
\end{equation}
\end{theorem}

\begin{proof}
In view of the dispersion analysis, using \eqref{eq:Hvf} with $v_h = \theta_p^j$ yields 
\begin{equation}
\omega_h^2 = \frac{a(\theta^j_p, u^h)}{b(\theta^j_p, u^h)} = \frac{A_p U_p /h }{B_p U_p h} = \frac{A_p U_p }{B_p U_p h^2 },
\end{equation}
which is known as the Rayleigh quotient. Applying Lemma \ref{lem:aoverb} and substituting $\Lambda = \mu h$ completes the proof.
\end{proof}

In the view of the duality principle, we have the following.
\begin{corollary}
For each eigenvalue, there holds
\begin{equation} \label{eq:eigenerr}
\lambda^h - \lambda = 2 (-1)^{p+1} \Big( \sum_{k=1}^p \frac{k^{2p+2}}{(2p+2)!} A^{j+k}_p  + \frac{k^{2p}}{(2p)!} B^{j+k}_p  \Big) \mu^{2p+2} h^{2p} + \mathcal{O}(h^{2p+2}).
\end{equation}
\end{corollary}

\begin{remark} \label{remk:2p}
This result validates that $| \lambda^h - \lambda | < Ch^{2p}.$ We can state this more explicitly with respect to $\mu$, thus the relative eigenvalue error $| \lambda^h - \lambda | / \lambda  < C(\mu h)^{2p}.$ Thus, the $2p$ order is also with respect to $\mu$, even though $\mu$ can be a large number. When $\mu$ is large, one requires $h$ to be sufficiently small for the bound on the error to be relevant. In other words, for the approximation to be relevant, we require that the product $\mu h$ remains strictly bounded.
\end{remark}

\section{Dispersion-minimized mass} \label{sec:dmm}
Section \ref{sec:disperr} establishes the optimal convergence order $2p$ for the dispersion error in 1D. The key identity for the analysis is the identity \eqref{eq:AkBk}, which limits the convergence order. To achieve a higher order of convergence we require that the identity \eqref{eq:AkBk} is to be satisfied for an increasing number of degrees of freedom $m$ beyond than $m=2,\cdots, p$. To establish the identity \eqref{eq:AkBk} for more values of $m$, we consider appropriate approximations of $A_p^{j+k}$ and $B_p^{j+k}$.  

From the view of Strang's lemma \cite{ciarlet1978finite}, for finite or isogeometric elements of \eqref{eq:pde} (with constant diffusion coefficient) in 1D on a uniform mesh, the stiffness matrix entries need to be exactly integrated by the quadrature rules which at least integrate exactly polynomials up to order $2p-2$. Since we consider \eqref{eq:pde} where the diffusion coefficient is a constant, the stiffness entries correspond to the integration of products of polynomials of order up to $2p-2$. Thus, the values of $A_p^{j+k}$ should remain the same. Therefore, we only consider approximations of the mass entries. This  motivates us to introduce the following dispersion-minimized mass.

\subsection{Local row-wise problem}
We first introduce the following linear system
\begin{equation}
\sum_{k=1}^n \frac{k^{2m}}{(2m)!} \alpha_k = \beta_m, \qquad m=1,2,\cdots, n,
\end{equation}
where $\alpha = (\alpha_1, \alpha_2, \cdots, \alpha_n)$ is a vector containing the unknowns and $\beta = (\beta_1, \beta_2, \cdots, \beta_n)$ is a given vector. We write this system in a matrix-vector form 
\begin{equation} \label{eq:dmmc}
\aleph  \alpha = \beta, 
\end{equation}
where 
\begin{equation}  \label{eq:mA}
\aleph = 
\begin{bmatrix}
\frac{1^2}{2!} & \frac{2^2}{2!} & \cdots & \frac{n^2}{2!} \\[0.2cm]
\frac{1^4}{4!} & \frac{2^4}{4!} & \cdots & \frac{n^4}{4!} \\[0.2cm] 
\vdots & \vdots & \ddots & \vdots \\[0.2cm]
\frac{1^{2n}}{(2n)!} & \frac{2^{2n}}{(2n)!} & \cdots & \frac{n^{2n}}{(2n)!} \\[0.2cm] 
\end{bmatrix},
\end{equation}
which is always invertible for any positive integer $n$. For simplicity, we also denote the following matrix
\begin{equation} \label{eq:mB}
\tilde \aleph = 
\begin{bmatrix}
\frac{1^4}{4!} & \frac{2^4}{4!} & \cdots & \frac{n^4}{4!} \\[0.2cm] 
\frac{1^6}{6!} & \frac{2^6}{6!} & \cdots & \frac{n^6}{6!} \\[0.2cm]
\vdots & \vdots & \ddots & \vdots \\[0.2cm]
\frac{1^{2n+2}}{(2n+2)!} & \frac{2^{2n+2}}{(2n+2)!} & \cdots & \frac{n^{2n+2}}{(2n+2)!} \\[0.2cm] 
\end{bmatrix},
\end{equation}
which is always invertible for any positive integer $n$. For the $p$-th order isogeometric elements with stiffness entries $\hat A_p = [A^{j+1}_{p} \quad A^{j+2}_{p} \quad \cdots \quad A^{j+p}_{p} ]$, that is a half of $A_p$ defined in \eqref{eq:UAB}, the local problem is to find 
\begin{equation}
\hat B_{p,O} = [B^{j+1}_{p,O} \quad B^{j+2}_{p,O} \quad \cdots \quad B^{j+p}_{p,O} ]
\end{equation}
 satisfying \eqref{eq:dmmc} in the form of
\begin{equation} \label{eq:dmmc0}
\aleph \hat B_{p,O} = - \tilde \aleph \hat A_p.
\end{equation}

Due to the non-singularity of the matrix $\aleph$, $B_{p,O}$ is uniquely solvable. Once $\hat B_{p,O}$ is obtained, 
we extend its definition to all relevant entries using the symmetry of the entries to the mass, that is, 
\begin{equation}
B^{j-k}_{p,O} = B^{j+k}_{p,O}, \qquad k = 1,2,\cdots, p.
\end{equation}
Due to the partition of unity of the B-spline basis functions, invoking mass conservation, we also define the middle entry 
\begin{equation}
B^j_{p,O} = 1- 2 \sum_{k=1}^p B^{j+k}_{p,O}.
\end{equation}

For $p=1,2,3,4$, these entries are listed in the right-most column of Table \ref{tab:AB}.
We call the mass entries 
\begin{equation} \label{eq:newmass}
B_{p,O} = [B^{j-p}_{p,O} \quad B^{j-p+1}_{p,O} \quad \cdots \quad B^{j+p}_{p,O} ]
\end{equation}
the dispersion-minimized mass entries as they render the minimal dispersion error, a result we derive in the following subsection.

\begin{remark}
The dispersion-minimized mass entries corresponding to the boundary elements are obtained in a similar fashion. For simplicity, we limit our discussion here to periodic boundary conditions. For non-uniform meshes and non-constant diffusion coefficient cases, the dispersion-minimized mass entries can be also obtained in a similar fashion. We leave this for future study as the analysis is more involved. The extension to multiple dimensions is presented in the Section \ref{sec:md}.
\end{remark}

\subsection{Minimized dispersion error}
In this section, we derive the minimized dispersion error for the mass entries $B_{p,O}$ defined in \eqref{eq:newmass}.
Firstly, we establish the following identity.
\begin{lemma} \label{lem:AkBkp+1}
For any positive integer $p$ with $B_{p,O}$ defined in \eqref{eq:dmmc0}, there holds
\begin{equation} 
\sum_{k=1}^p \Big( \frac{k^{2m}}{(2m)!} A_p^{j+k}  + \frac{k^{2m-2}}{(2m-2)!} B_{p,O}^{j+k} \Big) = 0, \qquad m=2,3,\cdots, p+1.
\end{equation}
\end{lemma}

\begin{proof}
Substituting \eqref{eq:mA} and \eqref{eq:mB} into the matrix system \eqref{eq:dmmc0}, we write it as a summation 
\begin{equation*}
\sum_{k=1}^p \frac{k^{2m}}{(2m)!} B_{p,O}^{j+k} = -\sum_{k=1}^p \frac{k^{2m+2}}{(2m+2)!} A_{p}^{j+k} , \qquad m = 1,2,\cdots, p.
\end{equation*}
Rewriting this equation completes the proof.
\end{proof}

The identity in Lemma \ref{lem:AkBkp+1} is true for $m=2,\cdots, p+1$, which is satisfied for one more equation than that in \eqref{eq:AkBk}. This extra identity for $m=p+1$ gives us superconvergence, which is an order of $2p+2$. We establish the minimized dispersion error as follows. 

\begin{theorem} \label{thm:dmmdisperr}
For each discrete mode  $\omega_h$, the discrete dispersion error is
\begin{equation} \label{eq:dmmdisperr}
\omega_h^2 - \mu^2 = 2 (-1)^p \Big( \sum_{k=1}^p \frac{k^{2p+4}}{(2p+4)!} A^{j+k}_p  + \frac{k^{2p+2}}{(2p+2)!} B^{j+k}_{p,O}  \Big) \mu^{2p+4} h^{2p+2} + \mathcal{O}(h^{2p+4}).
\end{equation}
\end{theorem}

\begin{proof}
Using the previous lemma and following the same type of arguments in Section \ref{sec:disperr} completes the proof.
\end{proof}

\begin{remark}
Theorem \ref{thm:dmmdisperr} further implies the superconvergence order of the eigenvalue error, that is 
\begin{equation} \label{eq:ee2more}
| \tilde \lambda^h_{p,O} - \lambda | \le C h^{2p+2},
\end{equation}
where $\tilde \lambda^h_{p,O}$ is the approximated eigenvalue when the dispersion-minimized mass is utilized. The dispersion error is minimized as it cannot be further reduced as there are no more degrees of freedom left on mass entries. In the next Section, we show this minimized dispersion error by optimal blending quadratures. Furthermore, we show that both the convergence orders and the leading order coefficients are the same.
\end{remark}

\subsection{Quadrature rules for dispersion-minimized mass} \label{sec:rule}
The local problem \eqref{eq:dmmc0} is a linear system of dimension $p$ for $p$-th order isogemetric elements. Due to the low dimension of the system, it is efficient to assemble these entries for the mass matrix. Despite the efficiency, we present in this subsection the quadrature rules for evaluating the dispersion-minimized mass entries. 

We develop a unified quadrature rule for evaluating both the stiffness and mass entries. 
Invoking the symmetry of the entries, we have in total $2p+2$ restrictions, that is $A_p^{j+k}$ and $B_{p,O}^{j+k}$ for $k=0,1,\cdots, p.$  They are, however, nonlinearly dependent. To construct the dispersion-minimized mass entries as well as the stiffness entries using numerical integration, we seek a unified quadrature rule which has the minimal number of points.

Let $N_p$ be the minimum number of quadrature points required for $p$-th order B-spline elements. The problem of finding the quadrature rule on the reference interval $[0,1]$  is to seek $\hat \varpi_{p,l}$ and $\hat n_{p,l}$ with $l=1,2, \cdots, N_p$ such that for a fixed B-spline basis function $\hat \theta_p^j$, there holds
\begin{equation} \label{eq:newquad}
\begin{aligned}
a_h(\hat \theta_p^j, \hat \theta_p^{j+k}) & = \sum_{l=1}^{N_p} \hat \varpi_{p,l} \nabla \big( \hat \theta_p^j ( \hat n_{p,l} )) \cdot \nabla \big( \hat \theta_p^{j+k} (\hat n_{p,l}) ) = A_{p}^{j+k}, \quad & k =0,1, \cdots p, \\
b_h(\hat \theta_p^j, \hat \theta_p^{j+k}) & = \sum_{l=1}^{N_p} \hat \varpi_{p,l} \hat \theta_p^j ( \hat n_{p,l} ) \hat \theta_p^{j+k} (\hat n_{p,l}) = B_{p,O}^{j+k}, \quad & k  =0,1, \cdots p,
\end{aligned}
\end{equation}
where $B_{p,O}^{j+k}$ is the solution of \eqref{eq:dmmc0}. 

\begin{remark}
We solve \eqref{eq:newquad} by symbolical calculation using \textbf{Mathematica}. From the definition \eqref{eq:ABdefinition}, both $A_p^{j+k}$ and $B_{p,O}^{j+k}$ are values already on the reference interval. These entries are listed in the first and the last columns in the Table \ref{tab:AB}. There are $2p+2$ restrictions, and each quadrature point has two degrees of freedom (the point location and its weight). Therefore, $N_p\le p+1$. For any $p$, we find the minimum number $N_p$ by trial and error, running from 1 point to at most $p+1$ points. In the case where the $2p+2$ restrictions can be reduced to an odd number of restrictions, we add a condition $\hat n_{p,1} = 0$. Thus, the resulting quadrature rule is of the Gauss-Radau type (see also \cite{bartovn2016optimal,bartovn2017gauss,kythe2004handbook}).
\end{remark}

We present the quadrature rules for $p=1,2,3$ as follows.
\begin{equation} \label{eq:quaddmm}
\begin{aligned}
p=1: \qquad \hat n_{1,1} & = \frac{1}{2} \pm \frac{\sqrt{6}}{6}, \qquad \hat \varpi_{1,1}  = 1, \\
p=2: \qquad \hat n_{2,1} & = 0, \quad \hat n_{2,2} = \frac{1}{2} \pm \frac{\sqrt{15}}{30}, \qquad \hat \varpi_{2,1}  = \frac{2}{7}, \quad \hat \varpi_{2,2}  = \frac{5}{7}, \\
p=3: \qquad \hat n_{3,1} & = 0, \quad \hat n_{3,2} = \frac{1}{2} \pm \frac{\sqrt{14}}{14}, \qquad \hat \varpi_{3,1}  = -\frac{17}{375}, \quad \hat \varpi_{3,2}  = \frac{392}{375}. \\
\end{aligned}
\end{equation}

\begin{remark}
The plus-minus $\pm$ specifies two different rules for $p=1,2,3$. For $p=2$, these rules do not fully integrate $C^0$ polynomials of order up to 3. For $p=3$, there is a negative weight. 
%For $p=4$,  the symbolical calculations show that we can find a two-point rule which satisfies all the conditions in  \eqref{eq:newquad}. However, the resulting points are complicated complex numbers. 
The rules for the boundary elements are different. Developing quadrature rules for higher order and boundary elements will be the subject of further investigation.  
\end{remark}

\section{Optimal blending in 1D} \label{sec:ob}
Section \ref{sec:disperr} establishes the eigenvalue error estimates when both the stiffness and the mass matrix entries are integrated exactly in 1D, which can be done using, for example $G_{p+1}$. Section \ref{sec:dmm} optimized the dispersion error by appropriately defining the mass entries. In this section, we develop the dispersion-minimized mass entries by optimally blending different quadrature rules, which generalizes the results of \cite{calo2017dispersion} from $p=1,2,\cdots, 7$ to arbitrary order $p$.

\subsection{Dispersion error when using other quadrature rules}
Section \ref{sec:ps} discusses that the rules $G_p$, $R_p$, and $L_{p+1}$ integrate $a(\theta_p^j, \theta_p^l)$ exactly and underintegrate $b(\theta_p^j, \theta_p^l)$ in 1D. 
Now, we denote by $Q_p$  any quadrature rule which integrates any polynomial of order $2p-2$  and by $O_p$ the optimal quadrature with minimal dispersion.
% but not higher than $2p-1$ exactly as $Q_p$. 
 In one dimension, since $a(\theta_p^j, \theta_p^l)$ is the integration of polynomials of order $2p-2$ while $b(\theta_p^j, \theta_p^l)$ is the integration of polynomials of order $2p$, $Q_p$ integrates $a(\theta_p^j, \theta_p^l)$ exactly and underintegrates $b(\theta_p^j, \theta_p^l)$. Both $G_p$ and $L_{p+1}$ are typical examples of such quadrature rules of type $Q_p$. There are  infinitely many such quadrature rules if one disregards the number of quadrature points. For blending, we assume here that $Q_p \ne O_p.$

Now, we denote 
\begin{equation}
\tilde A^{j-k}_{p,Q} = a_h(\theta^{j-k}_p, \theta^j_p) h, \qquad \tilde B^{j-k}_{p,Q} = b_h(\theta^{j-k}_p, \theta^j_p) / h,
\end{equation}
where $Q$ specifies a quadrature rule applied, which can be set to $G_p, R_p, L_{p+1}$, or generically to $Q_p$. In one dimension, one immediately has 
\begin{equation} \label{eq:aat}
\tilde A^{j-k}_{p,Q_p} = A^{j-k}_{p}
\end{equation}
and thus,  we use them interchangeably in the discussion. For $p=1,2,3,4$, the values of $\tilde B^{j-k}_{p,Q}$ where $Q=G_p, R_p, L_{p+1}$ are listed in Table \ref{tab:AB}.

\begin{table}[ht]
\centering 
\begin{tabular}{| c | c || c | c | c | c | c | c| }
\hline
%degree & mesh & $\lambda_1, \eta=1$ & $\lambda_1, \eta=2p+3$ & $\lambda_3, \eta=1$ & $\lambda_3, \eta=2p+3$ \\[0.1cm] 
$p$ & $k$ & $A_p^{j+k}$ & $B_p^{j+k}$   & $\tilde B_{p,G_p}^{j+k}$ & $\tilde B_{p,L_{p+1}}^{j+k}$ & $\tilde B_{p,R_p}^{j+k}$ & $\tilde B_{p,O_p}^{j+k}$  \\[0.1cm] \hline

   & 0 & 2 & $\frac{2}{3}$ & $\frac{1}{2}$ & 1 & 1 & $\frac{5}{6}$ \\[0.2cm]
1 & 1 & $-1$ & $\frac{1}{6}$ & $\frac{1}{4}$ & 0 & 0 & $\frac{1}{12}$ \\[0.2cm] \hline

 & 0 & 1 & $\frac{11}{20}$ & $\frac{13}{24}$ & $\frac{9}{16}$ & $\frac{5}{9}$ & $\frac{67}{120}$ \\[0.2cm]
2 & 1 & $-\frac{1}{3}$ & $\frac{13}{60}$ & $\frac{2}{9}$ & $\frac{5}{24}$ & $\frac{23}{108}$ & $\frac{19}{90}$ \\[0.2cm]
   & 2 & $-\frac{1}{6}$ & $\frac{1}{120}$ & $\frac{1}{144}$ & $\frac{1}{96}$ & $\frac{1}{108}$ & $\frac{7}{720}$ \\[0.2cm] \hline

 & 0 & $\frac{2}{3}$ & $\frac{151}{315}$ & $\frac{23}{48}$ & $\frac{259}{540}$ & $\frac{863}{1800}$ & $\frac{3629}{7560}$ \\[0.2cm]
   & 1 & $-\frac{1}{8}$ & $\frac{397}{1680}$ &$\frac{227}{960}$ & $\frac{17}{72}$ & $\frac{189}{800}$ & $\frac{2377}{10080}$ \\[0.2cm]
3   & 2 & $-\frac{1}{5}$ & $\frac{1}{42}$ & $\frac{19}{800}$& $\frac{43}{1800}$ & $\frac{143}{6000}$ & $\frac{121}{5040}$ \\[0.2cm]
   & 3 & $-\frac{1}{120}$ & $\frac{1}{5040}$ &$\frac{1}{4800}$ & $\frac{1}{5400}$ & $\frac{7}{36000}$ & $\frac{1}{6048}$ \\[0.2cm] \hline
   
 & 0 & $\frac{35}{72}$ & $\frac{15619}{36288}$ & $\frac{52063}{120960}$ & $\frac{41651}{96768}$ & $\frac{91111}{211680}$ & $\frac{156211}{362880}$ \\[0.2cm]
   & 1 & $-\frac{11}{360}$ & $\frac{44117}{181440}$ & $\frac{73529}{302400}$ & $\frac{29411}{120960}$ & $\frac{514697}{2116800}$ & $\frac{220543}{907200}$ \\[0.2cm]
4   & 2 & $-\frac{17}{90}$ & $\frac{913}{22680}$ & $\frac{1739}{43200}$ & $\frac{9739}{241920}$ & $\frac{42607}{1058400}$ & $\frac{36541}{907200}$ \\[0.2cm]
   & 3 & $-\frac{59}{2520}$ & $\frac{251}{181440}$ &$\frac{2929}{2116800}$ & $\frac{1171}{846720}$ & $\frac{20497}{14817600}$ & $\frac{1249}{907200}$ \\[0.2cm]
   & 4 & $-\frac{1}{5040}$ & $\frac{1}{362880}$ & $\frac{23}{8467200}$ & $\frac{19}{6773760}$ & $\frac{41}{14817600}$ & $\frac{13}{3628800}$ \\[0.2cm] \hline   

\hline
\end{tabular}
\caption{Stiffness and mass entries $A_p^{j+k}$ and $B_p^{j+k}$ as well as the approximated mass entries $\tilde B_{p,Q}^{j+k}$ for $Q=G_p, L_{p+1}, R_p, O_p$. The entries in the last column of $\tilde B_{p,O_p}^{j+k}$ are also the dispersion-minimized mass entries.} 
\label{tab:AB} 
\end{table}

To derive the dispersion error when we apply $Q_p$, we first present the following Postulate and Theorem.
\begin{postulate} \label{pos:tpbsplines}
For any positive integer $p>1$ and $m=2,3,\cdots, p$, there holds
\begin{equation}
\sum_{k=1}^p \Big( \frac{k^{2m}}{(2m)!} \tilde A_{p,Q_p}^{j+k}  + \frac{k^{2m-2}}{(2m-2)!} \tilde B_{p,Q_p}^{j+k} \Big) = 0.
\end{equation}
\end{postulate}

\begin{remark}
This Postulate generalizes Lemma \ref{lem:AkBk}. The result is verified for $p=1,2,3,4$ using the values listed in Table \ref{tab:AB}. For an arbitrary quadrature rule $Q_p$ and any order $p$, the proof is an open question and will be the subject of future work.
\end{remark}

\begin{theorem} \label{thm:Qperr}
For each eigenvalue, there holds
\begin{equation} \label{eq:teigenerr}
\tilde \lambda^h_{p,Q_p} - \lambda = 2 (-1)^{p+1} \Big( \sum_{k=1}^p \frac{k^{2p+2}}{(2p+2)!} \tilde A^{j+k}_{p,Q_p}  + \frac{k^{2p}}{(2p)!} \tilde B^{j+k}_{p,Q_p}  \Big) \mu^{2p+2} h^{2p} + \mathcal{O}(h^{2p+2}),
\end{equation}
where $\tilde \lambda^h_{p,Q_p}$ denotes the approximated eigenvalue while $Q_p$ is applied.
\end{theorem}

\begin{proof}
Applying Postulate \ref{pos:tpbsplines} and following the same type of arguments in Section \ref{sec:disperr} with $B_p$ replaced by $\tilde B_{p, Q_p}$, we complete the proof.
\end{proof}

\subsection{Optimal blending}
Assume that $Q_p \ne G_{p+1}$. $Q_p$ does not integrate the mass entries exactly in 1D.  The differences in the leading coefficients of \eqref{eq:eigenerr} and \eqref{eq:teigenerr} allow us to blend different quadratures to remove the leading order terms from the error estimates. From the insights on the lower order cases as done in \cite{calo2017dispersion}, we can consider the following blending quadrature rule
\begin{equation} \label{eq:taubr}
Q_\tau = \tau G_{p+1} + (1-\tau) Q_p,
\end{equation}
where $\tau$ is the blending parameter. Now we have the following results for optimal blending coefficient.

\begin{lemma} \label{lem:opbsplines}
Let 
\begin{equation} \label{eq:tau}
\tau = \frac{\sum_{k=1}^p \frac{k^{2p+2}}{(2p+2)!} A^{j+k}_{p,Q_p}  + \frac{k^{2p}}{(2p)!} \tilde B^{j+k}_{p,Q_p}}{ \sum_{k=1}^p \frac{k^{2p}}{(2p)!} ( \tilde B^{j+k}_{p,Q_p} -B^{j+k}_{p} )  }.
\end{equation}
Then for any positive integer $p$, there holds
\begin{equation} \label{eq:abp1}
\sum_{k=1}^p \frac{k^{2m}}{(2m)!} \tilde A_{p,Q_\tau}^{j+k}  + \frac{k^{2m-2}}{(2m-2)!} \tilde B_{p,Q_\tau}^{j+k}  = 0
\end{equation}
for $m=2,3,\cdots p+1$.
\end{lemma}

\begin{proof}
Applying the blending rule \eqref{eq:taubr} yields % with \eqref{eq:aat} 
\begin{equation}
\begin{aligned}
\tilde A_{p,Q_\tau}^{j+k} & = \tau A_{p}^{j+k} + (1-\tau) \tilde A_{p,Q_p}^{j+k} , \\ % = A_{p}^{j+k}, \\
\tilde B_{p,Q_\tau}^{j+k} & = \tau B_{p}^{j+k} + (1-\tau) \tilde B_{p,Q_p}^{j+k}. \\
\end{aligned}
\end{equation}
Thus
\begin{equation} 
\begin{aligned}
\Xi & = \sum_{k=1}^p \Big( \frac{k^{2m}}{(2m)!} \tilde A_{p,Q_\tau}^{j+k}  + \frac{k^{2m-2}}{(2m-2)!} \tilde B_{p,Q_\tau}^{j+k} \Big) \\
& = \sum_{k=1}^p \frac{k^{2m}}{(2m)!} \big( \tau A_{p}^{j+k} + (1-\tau) \tilde A_{p,Q_p}^{j+k} \big) + \sum_{k=1}^p \frac{k^{2m-2}}{(2m-2)!} \big( \tau B_{p}^{j+k} + (1-\tau) \tilde B_{p,Q_p}^{j+k} \big)  \\
& = \tau \sum_{k=1}^p \Big( \frac{k^{2m}}{(2m)!} A_{p}^{j+k} + \frac{k^{2m-2}}{(2m-2)!} B_{p}^{j+k} \Big) + (1-\tau) \sum_{k=1}^p \Big( \frac{k^{2m}}{(2m)!}  \tilde A_{p,Q_p}^{j+k} + \frac{k^{2m-2}}{(2m-2)!} \tilde B_{p,Q_p}^{j+k}\Big). \\
\end{aligned}
\end{equation}
For $m=2,3,\cdots p,$ applying Lemmas \ref{lem:Ak} and \ref{lem:AkBk} gives 
\begin{equation}
\Xi = \tau \cdot 0 + (1-\tau) \cdot 0 = 0.
\end{equation}
For $m=p+1$, invoking $\tau$ with \eqref{eq:tau}, we obtain 
\begin{equation} 
\begin{aligned}
\Xi & = \sum_{k=1}^p \frac{k^{2p+2}}{(2p+2)!} A^{j+k}_{p,Q_p}  + \frac{k^{2p}}{(2p)!} \big( \tilde B^{j+k}_{p,Q_p} + \tau (B^{j+k}_{p} - \tilde B^{j+k}_{p,Q_p} ) \big) \\
& = \sum_{k=1}^p \frac{k^{2p+2}}{(2p+2)!} A^{j+k}_{p,Q_p}  + \frac{k^{2p}}{(2p)!} \tilde B^{j+k}_{p,Q_p}  + \tau \sum_{k=1}^p \frac{k^{2p}}{(2p)!} (B^{j+k}_{p} - \tilde B^{j+k}_{p,Q_p} ) \\
& = 0.
\end{aligned}
\end{equation}
This completes the proof.
\end{proof}

\begin{theorem} \label{thm:ob}
Let $\tau$ be defined as \eqref{eq:tau}.
Then for each eigenvalue, there holds
\begin{equation} \label{eq:oeigenerr}
\begin{aligned}
\tilde \lambda^h_{p,O_p} - \lambda & = 2 (-1)^{p} \Big( \sum_{k=1}^p \frac{k^{2p+4}}{(2p+4)!} A^{j+k}_{p}  + \frac{k^{2p+2}}{(2p+2)!} ( \tau B^{j+k}_{p} + (1-\tau) \tilde B^{j+k}_{p,Q_p} )  \Big) \mu^{2p+4} h^{2p+2} \\ 
& \quad + \mathcal{O}(h^{2p+4}).
\end{aligned}
\end{equation}
\end{theorem}

\begin{proof}

Invoking Lemma \ref{lem:opbsplines}, applying \eqref{eq:aat}, and following the arguments we describe in Section \ref{sec:disperr} with $B_p$ substituted by $\tilde B_{p, Q_\tau}$ completes the proof.
\end{proof}

One can optimally blend other quadrature rules similarly. We denote the following blendings for $Q=G_{p+1}, G_p, L_{p+1}, R_p$
\begin{equation} \label{eq:alltau}
\begin{aligned}
\tau_{gg} G_{p+1} + (1-\tau_{gg}) G_p,  & \qquad
\tau_{gl} G_{p+1} + (1-\tau_{gl}) L_{p+1}, \\
\tau_{gr} G_{p+1} + (1-\tau_{gr}) R_p, & \qquad
\tau_{pl} G_{p} + (1-\tau_{pl}) L_{p+1}, \\
\tau_{pr} G_{p} + (1-\tau_{pr}) R_p, & \qquad
\tau_{lr} L_{p+1} + (1-\tau_{lr}) R_p, \\
\end{aligned}
\end{equation}

Table \ref{tab:tau} shows these blending parameters. We cannot blend $L_{p+1}$ and $R_p$ for $p=1$ as both of them lead to the same mass entries. For higher order $p$ and other quadrature rules, the blending parameters are derived using \eqref{eq:tau} in a similar fashion. 
\begin{table}[ht]
\centering 
\begin{tabular}{| c || c | c | c | c | c | c| }
\hline
%degree & mesh & $\lambda_1, \eta=1$ & $\lambda_1, \eta=2p+3$ & $\lambda_3, \eta=1$ & $\lambda_3, \eta=2p+3$ \\[0.1cm] 
$p$ & $\tau_{gg}$ & $\tau_{gl}$   & $\tau_{gr}$ & $\tau_{pl}$ & $\tau_{pr}$ & $\tau_{lr}$  \\[0.1cm] \hline

1  & 2 & $\frac{1}{2}$ & $\frac{1}{2}$ & $\frac{1}{3}$ & $\frac{1}{3}$ & - - \\[0.1cm] \hline
2  & 2 & $\frac{1}{3}$ & $-\frac{1}{2}$ & $\frac{1}{5}$ & $-\frac{1}{5}$ & $\frac{2}{5}$ \\[0.1cm] \hline
3  &  $\frac{13}{3}$ & $-\frac{3}{2}$ & $-\frac{22}{3}$ & $-\frac{6}{7}$ & $-\frac{44}{21}$ & $\frac{22}{7}$ \\[0.1cm] \hline
4  &  $22$ & $-\frac{79}{5}$ & $-\frac{145}{2}$ & $-\frac{79}{9}$ & $-\frac{145}{9}$ & $\frac{580}{27}$ \\[0.1cm] \hline
\end{tabular}
\caption{Optimal blending parameters for various quadratures.} 
\label{tab:tau} 
\end{table}

\begin{remark}
The parameter $\tau$ defined in \eqref{eq:tau} delivers superconvergence on the eigenvalue errors.  This is the optimal blending parameter as it provides the best possible blending for reducing the dispersion errors. This blending is not limited to combining  $G_{p+1}$ and $Q_{p}$. One can find the optimal blending rule for two different $Q_p$s and all these different optimal blending rules lead to the same error expansion. Moreover, we point out that the mass entries $\tilde B^{j+k}_{p,Q_\tau}$ where $\tau$ is defined in \eqref{eq:tau}, that is, the mass entries of the optimal blending rule, are the same as those of the dispersion-minimized mass of Section \ref{sec:dmm}.
\end{remark}

Theorem \ref{thm:ob} establishes an error estimation for the eigenvalues when we apply the blended quadrature rules 
\begin{equation} \label{eq:2p+2}
| \tilde \lambda^h_{p,O_p} - \lambda | \le C h^{2p+2},
\end{equation}
which is the same as \eqref{eq:ee2more} in Section \ref{sec:dmm}.

It is not possible to combine more quadrature rules to deliver higher order convergence. From the discussions in Section \ref{sec:dmm}, $2p+2$ is the best one can obtain as there are no more degrees of freedom left for the mass entries. Alternatively, the following arguments confirm this statement. 

We consider blending of three different quadrature rules $Q_p^1, Q_p^2$, and $Q_p^3$  such that their corresponding leading terms of the error expansions are different. Theorem \ref{thm:Qperr} allows us to present their error expansions with one more term 
\begin{equation}
\tilde \lambda^h_{p,Q_p^m} - \lambda = T_{2p}^m \mu^{2p+2} h^{2p} + T_{2p+2}^m \mu^{2p+4} h^{2p+2} + \mathcal{O}(h^{2p+4}),
\end{equation}
where 
\begin{equation}
\begin{aligned}
T_{2p}^m & = 2 (-1)^{p+1} \sum_{k=1}^p \frac{k^{2p+2}}{(2p+2)!} \tilde A^{j+k}_{p,Q_p^m}  + \frac{k^{2p}}{(2p)!} \tilde B^{j+k}_{p,Q_p^m}, \\
T_{2p+2}^m & = 2 (-1)^{p} \sum_{k=1}^p \frac{k^{2p+4}}{(2p+4)!} \tilde A^{j+k}_{p,Q_p^m}  + \Big( \frac{k^{2p+2}}{(2p+2)!} + \frac{k^2 }{2!} T_{2p}^m \Big) \tilde B^{j+k}_{p,Q_p^m}. \\
\end{aligned}
\end{equation}
for $m=1,2,3.$ The blending of these three quadrature rules is expressed as
\begin{equation}
Q_{\tau}^3 = \tau_1 Q_p^1 + \tau_2 Q_p^2 + (1 - \tau_1 -\tau_2) Q_p^3.
\end{equation}
All $Q_p^1, Q_p^2, Q_p^3,$ and  $Q_{\tau}^3$ fully integrate the stiffness entries. 
Following the previous arguments, one obtains the error expansion below
\begin{equation}
\begin{aligned}
\tilde \lambda^h_{p,Q_\tau^3} - \lambda & = T_{2p}^O \mu^{2p+2} h^{2p} + T_{2p+2}^O \mu^{2p+4} h^{2p+2} + \mathcal{O}(h^{2p+4}), \\
%& = \big( \tau_1 T_{2p}^1 + \tau_2 T_{2p}^2 + (1 - \tau_1 - \tau_2) T_{2p}^3 \big)  \mu^{2p+2} h^{2p} \\
%& \quad + \big( \tau_1 T_{2p+2}^1 + \tau_2 T_{2p+2}^2 + (1 - \tau_1 - \tau_2) T_{2p+2}^3 \big) \mu^{2p+4} h^{2p+2} + \mathcal{O}(h^{2p+4}). \\
\end{aligned}
\end{equation}
where
\begin{equation}
\begin{aligned}
T_{2p}^O & = \tau_1 T_{2p}^1 + \tau_2 T_{2p}^2 + (1 - \tau_1 - \tau_2) T_{2p}^3 \\
& =  2 (-1)^{p+1} \sum_{k=1}^p \frac{k^{2p+2}}{(2p+2)!} A^{j+k}_{p}  + \frac{k^{2p}}{(2p)!} \Big( \tau_1 \tilde B^{j+k}_{p,Q_p^1} + \tau_2 \tilde B^{j+k}_{p,Q_p^2} + (1-\tau_1-\tau_2) \tilde B^{j+k}_{p,Q_p^3} \Big), \\
T_{2p+2}^O & = 2 (-1)^{p} \sum_{k=1}^p \Big( \frac{k^{2p+4}}{(2p+4)!} A^{j+k}_{p}  \\
& \quad + \big( \frac{k^{2p+2}}{(2p+2)!} + \frac{k^2 }{2!} T_{2p}^O \big) \big( \tau_1 \tilde B^{j+k}_{p,Q_p^1} + \tau_2 \tilde B^{j+k}_{p,Q_p^2} + (1-\tau_1-\tau_2) \tilde B^{j+k}_{p,Q_p^3} \big) \Big).
\end{aligned}
\end{equation}

However, the system 
\begin{equation} \label{eq:tau1tau2}
\begin{aligned}
T_{2p}^O & = 0 \\
T_{2p+2}^O & = 0 \\
\end{aligned}
\end{equation}
has no solution. Using $T_{2p}^O=0$ for $T_{2p+2}^O$, the system \eqref{eq:tau1tau2} reduces to
\begin{equation}
\begin{aligned}
\alpha_{1,1} \tau_1 + \alpha_{1,2} \tau_2 + \alpha_{1,3} (1 - \tau_1 - \tau_2) + \beta_1 & = 0, \\
\alpha_{2,1} \tau_1 + \alpha_{2,2} \tau_2 + \alpha_{2,3} (1 - \tau_1 - \tau_2) + \beta_2 & = 0, \\
\end{aligned}
\end{equation}
where 
\begin{equation}
\begin{aligned}
\alpha_{q,m} & = \sum_{k=1}^p \frac{k^{2p+2q-2}}{(2p+2q-2)!} \tilde B^{j+k}_{p,Q_p^m}, \quad && q  =1, 2, m=1,2,3,\\
\beta_q & = \sum_{k=1}^p \frac{k^{2p+2q}}{(2p+2q)!} A^{j+k}_{p}, && q  = 1,2.
\end{aligned}
\end{equation}
Given that  
\begin{equation} \label{eq:tau1tau2cond}
\begin{aligned}
\frac{ \alpha_{1,1} - \alpha_{1,3} }{ \alpha_{2,1} - \alpha_{2,3} } & = \frac{ \alpha_{1,2} - \alpha_{1,3} }{ \alpha_{2,2} - \alpha_{2,3} } ,\\
\alpha_{1,3} + \beta_1 & \ne ( \alpha_{2,3} + \beta_2 ) \frac{ \alpha_{1,1} - \alpha_{1,3} }{ \alpha_{2,1} - \alpha_{2,3} },
\end{aligned}
\end{equation}
the system \eqref{eq:tau1tau2} has no solution. Verifying \eqref{eq:tau1tau2cond} for arbitrary $p$ and $Q_p^m$ is necessary and will be the subject of future efforts. The condition \eqref{eq:tau1tau2cond} can be verified easily for special cases. For instance, setting $Q_p^1 = G_{p+1}, Q_p^2 = L_{p+1}$, and $Q_p^3 = G_p$, we have the following simplified systems
\begin{equation}
\begin{aligned}
2 \tau_1 + 5 \tau_2 & = 4, \\
14 \tau_1 + 35 \tau_2 & = 50, \\
\end{aligned}
\end{equation}
and
\begin{equation}
\begin{aligned}
3 a + 7 b & = 13, \\
3 a + 7 b & = 63, \\
\end{aligned}
\end{equation}
for $p=2$ and $p=3$, respectively. These systems do not have solutions. Therefore, one cannot increase the convergence orders by blending more than two quadrature rules. Alternatively, one can explain this limitation from the maximum number of unknowns for the mass entries as discussed in Section \ref{sec:dmm}.  For a fixed set of $A_p^{j+k}$, there are $p$ unknowns in the mass entries, that is $\tilde B_p^{j+k}$ with $k=1,2,\cdots, p$. In the optimal blending case, the identity \eqref{eq:abp1} is satisfied for $p$ equations, that is $m=2,3,\cdots, p+1$, thus there are no degrees of freedom left on $\tilde B_p$ for the \eqref{eq:abp1} to be satisfied for $m=p+2$, which prevents us from obtaining a convergence of order $2p+4$.

\section{Extension to multidimension and eigenfunction error estimates} \label{sec:md} 
The analysis of generalization to multidimension is studied in the literature for tensor-product basis functions when using finite elements in \cite{ainsworth2010optimally} and when using isogeometric elements \cite{calo2017dispersion}. From these references, the multidimensional problem admits a nontrivial solution provided that 
\begin{equation} 
\omega^2 = \sum_{k=1}^d  \omega_k^2, 
\end{equation}
or alternatively in the eigenvalue form is 
\begin{equation} \label{eq:1dmd}
\lambda^h = \sum_{k=1}^d \lambda^h_k,
\end{equation}
where $d$ is the dimension and $\lambda^h_k = \omega^2_k$ being the approximated wave frequencies squared.
This implies that the optimal blending for the one-dimensional case extends to the arbitrary dimension and is independent of the number of spatial dimensions. We deduce the corresponding optimized dispersion error expression for multidimensional problems from \eqref{eq:2p+2} and \eqref{eq:1dmd}, which is
\begin{equation} \label{eq:mdo}
| \tilde \lambda^h_{O_p} - \lambda | = C h^{2p+2}.
\end{equation}

We now establish the error estimate for the eigenfunctions in the same fashion as in \cite{calo2017dispersion}. The following theorem establishes the eigenfunction errors.  The work  \cite{calo2017dispersion} established the theorem with a complete proof for isogeometric polynomial order up to $p=7$. We refer to  \cite{calo2017dispersion} for a proof of the following theorem which is a simple extension. 

\begin{theorem} \label{thm:efe}
For a fixed discrete eigenmode, assume that the eigenfunction $u$ and $\tilde u^h$ are normalized, that is, $b(u, u)=1$ and $\tilde b_h(\tilde u^h, \tilde u^h) = 1$, and the signs of eigenfunctions of $u$ and $\tilde u^h$ are chosen such that $b(u, \tilde u^h) > 0$. Then for sufficiently small $h$, we have the estimate
\begin{equation}
\| u - \tilde u^h_Q \|_E \le C h^p,
\end{equation}
where $\| \cdot \|_E$ is the energy norm and $Q$ specifies a quadrature rule $G_{p+1}$, $O_p$, or $Q_p$.
\end{theorem}

\begin{figure}[ht]
\centering
%\includegraphics[height=7.0cm]{d.eps}
%\caption{Basis functions and their derivatives for $C^1$ quadratic isogeometric analysis.}
\label{fig:d}
\end{figure}

%\section{Non-uniform mesh and non-constant coefficient}
%\textcolor{blue}{This idea may work for non-uniform mesh and non-constant diffusion coefficient, which is subject to further study} One way to propose $B_p$ is to satisfy these equations. Another way is to find integration rules satisfying exact integration for $2p-2$ and optimize the last equation when $m=p+1$.

%Radau Quadrature

\section{Numerical examples} \label{sec:num}
In this section, we present the numerical simulations of the problem \eqref{eq:pde} in one and two dimensions 
%with periodic boundary conditions 
using the dispersion-minimized mass for $C^1$ quadratic and $C^2$ cubic isogeometric elements on uniform meshes.  Both our symbolic and numerical calculations show that the dispersion-minimized mass, the quadrature rules \eqref{eq:quaddmm}, and the optimally-blended quadrature rules (see Table \ref{tab:tau}) yield  the same stiffness and mass entries on uniform meshes in both one and multiple dimensions. We utilize the quadrature rules \eqref{eq:quaddmm} for the following numerical experiments. 

For numerical simulations using higher order isogeometric elements,
%or simulations of the problem \eqref{eq:pde} with non-periodic boundary conditions, 
we refer to \cite{calo2017dispersion}, where optimally blended $G_{p+1}, G_p$, and $L_{p+1}$ quadrature rules are studied for $p=1,2,\cdots 7$. 
%We note that in one dimension $G_{p+1}$ integrates exactly polynomials up to order $2p+1$ while both $G_p$ and $L_{p=1}$ integrates exactly polynomials up to order $2p-1$. 
In this paper, the concept of the optimal blending is extended to Radau and other general rules. In one dimension, $R_p$ integrates exactly polynomials up to order $2p-2$, which is less than $2p-1$ as for $G_p$ and $L_{p+1}$. 
%However, the $G_p$-$R_p$ (or $L_{p+1}$-$R_p$) optimally blended rule yield the same stiffness and mass entries in both one and two dimensions for \eqref{eq:pde} with periodic boundary conditions on uniform meshes. 
For comparison, we also present the numerical results while using $R_p$. 

\begin{figure}[ht]
\centering
\includegraphics[height=5.8cm]{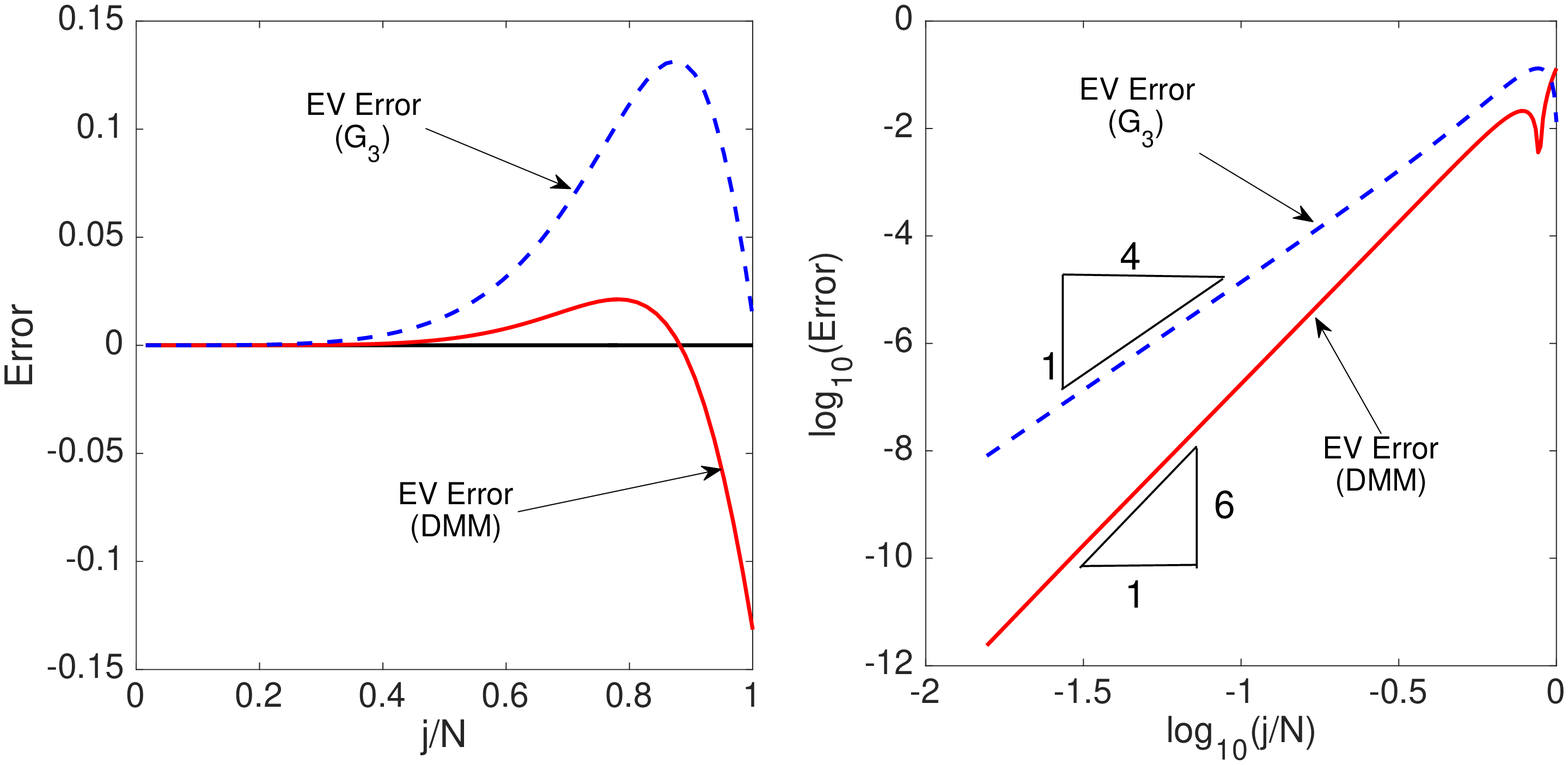} 
\caption{Relative eigenvalue (EV) errors for $C^1$ quadratic isogeometric analysis with fully integrated mass ($G_3$) and dispersion-minimized mass (DMM) in 1D. }
\label{fig:p2eve}
\includegraphics[height=5.8cm]{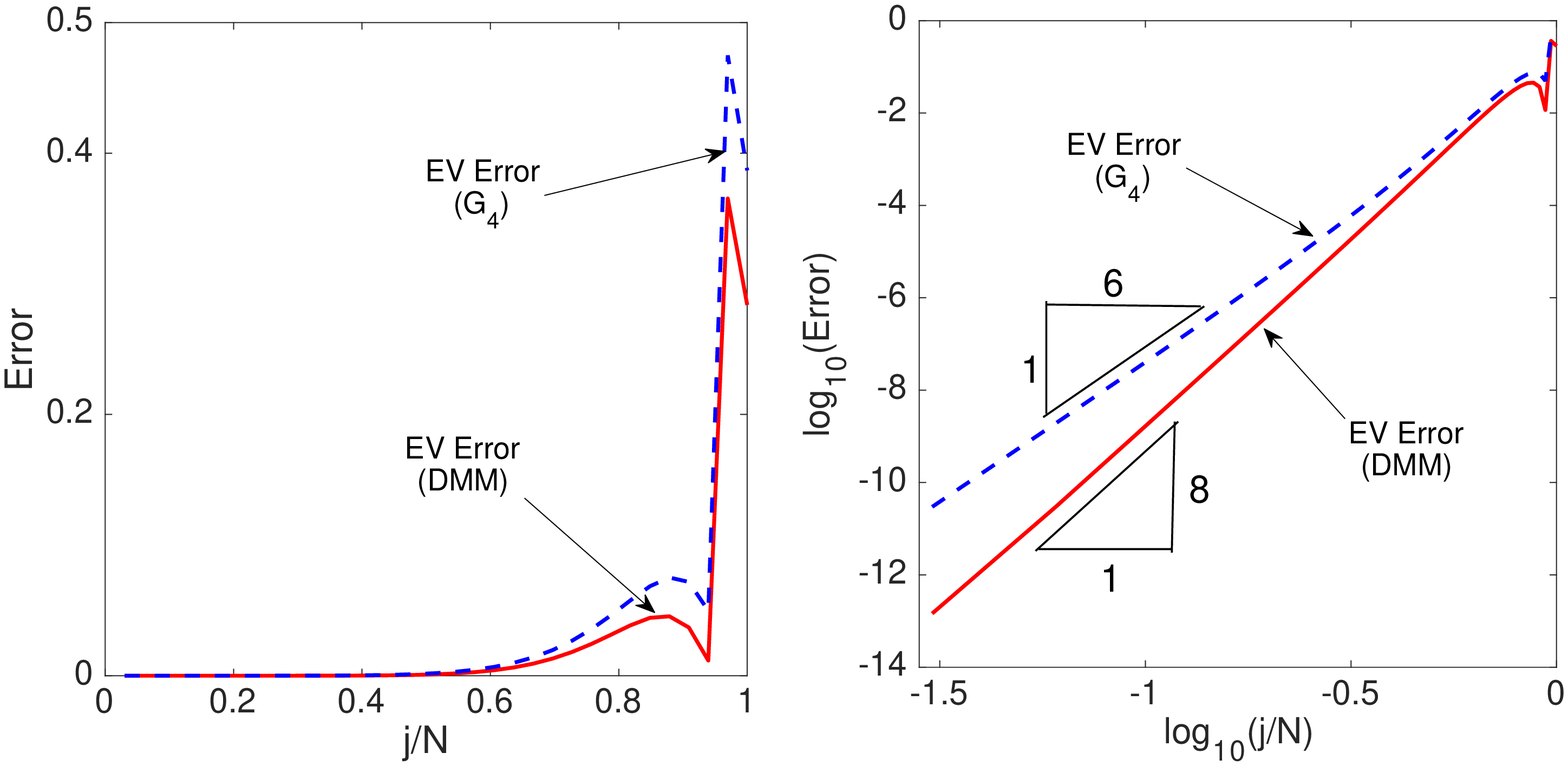} 
\caption{Relative eigenvalue (EV) errors for $C^2$ cubic isogeometric analysis with fully integrated mass ($G_4$) and dispersion-minimized mass (DMM) in 1D. }
\label{fig:p3eve}
\end{figure}

We assume that once the eigenvalue problem is solved, the numerical approximations to the eigenvalues are sorted in ascending order and paired with the true eigenvalues. We focus on the numerical approximation properties of the eigenvalues. In the following, however, we report the relative eigenvalue (EV)  errors  as well as the eigenfunction (EF) errors in energy norm.

\subsection{Numerical study in 1D}
We consider $\Omega = [0, 1]$. The one dimensional differential eigenvalue problem \eqref{eq:pde} %with periodic homogeneous boundary conditions 
has true eigenvalues and eigenfunctions 
\begin{equation}
\lambda_j = j^2 \pi^2, \quad \text{and} \quad u_j = \sqrt{2} \sin( j\pi x), \quad j = 1, 2, \cdots,
\end{equation}
respectively. Figures \ref{fig:p2eve} and \ref{fig:p3eve} show the relative eigenvalue errors, defined as $\frac{\lambda_j^h - \lambda_j}{\lambda_j}$, of $C^1$ quadratic and $C^2$ cubic isogeometric approximations, respectively. The isogeometric mass entries are fully integrated by the Gauss rule $G_{p+1}$ and compared with the dispersion-minimized mass. The meshes are uniform and the mesh size for $C^1$ quadratic isogeometric elements is $1/64$ while $1/32$ for the cubic case. For both $p=2$ and $p=3$, the dispersion-minimized mass leads to smaller eigenvalue errors and their convergence rates are of order $2p+2$, which is two extra order of convergence than those of the fully integrated cases. These convergence rates shown in Figures \ref{fig:p2eve} and \ref{fig:p3eve} are with respect to the wave numbers as we fixed the mesh size $h$. These rates confirm the discussion of Remark \ref{remk:2p} in terms of the wave numbers. For eigenfunctions, Figures \ref{fig:p2efe} and \ref{fig:p3efe} show the energy norm eigenfunction errors of $C^1$ quadratic and $C^2$ cubic isogeometric approximations, respectively. The errors are of optimal convergence order $p$. 

\begin{table}[ht]
\centering 
\begin{tabular}{| c | c | c | c | c | c | c | c| c| c| c| c|}
\hline
%\multicolumn{10}{|c|}{$p=2$}  \\[0.1cm] \hline
 \multicolumn{2}{|c|}{Set} &  \multicolumn{3}{c|} {$ |\lambda_1^h - \lambda_1| / \lambda_1$ }  & \multicolumn{3}{c|} {$ |\lambda_2^h - \lambda_2| /  \lambda_2$}  & \multicolumn{3}{c|} {$ |\lambda_4^h - \lambda_4| /  \lambda_4$}   \\[0.1cm] \hline
$p$ &  $N$ & $G_{p+1}$ & $R_p$ & DMM & $G_{p+1}$ & $R_p$ & DMM & $G_{p+1}$ & $R_p$ & DMM  \\[0.1cm] \hline

	&  8	& 3.4e-5	& 3.6e-6	& 6.7e-7	& 6.0e-4	& 8.3e-5	& 4.3e-5	& 1.3e-2	& 2.9e-3	& 2.8e-3 \\[0.1cm]
	& 16	& 2.1e-6	& 4.5e-7	& 1.0e-8	& 3.4e-5	& 7.7e-6	& 6.7e-7	& 6.0e-4	& 1.6e-4	& 4.3e-5 \\[0.1cm]
2	& 32	& 1.3e-7	& 3.5e-8	& 1.6e-10	& 2.1e-6	& 5.8e-7	& 1.0e-8	& 3.4e-5	& 9.8e-6	& 6.7e-7 \\[0.1cm]
	& 64	& 8.1e-9	& 2.4e-9	& 2.4e-12	& 1.3e-7	& 3.9e-8	& 1.6e-10	& 2.1e-6	& 6.4e-7	& 1.0e-8 \\[0.1cm] \hline
 \multicolumn{2}{|c|}{$\rho_2$}  & 4.0 & 3.5 & 6.0 & 4.1 & 3.7 & 6.0 & 4.2 & 4.0 & 6.0 \\[0.1cm] \hline
 
	& 4	& 9.7e-6	& 8.8e-6	& 1.7e-6	& 9.9e-4	& 9.3e-4	& 4.5e-4	& 2.4e-1	& 1.8e-1	& 1.9e-1 \\[0.1cm]
	& 8	& 1.3e-7	& 1.2e-7	& 7.3e-9	& 1.0e-5	& 9.1e-6	& 2.0e-6	& 1.1e-3	& 1.1e-3	& 5.6e-4 \\[0.1cm]
3	& 16	& 1.9e-9	& 1.7e-9	& 2.9e-11	& 1.3e-7	& 1.2e-7	& 7.6e-9	& 1.0e-5	& 9.2e-6	& 2.1e-6 \\[0.1cm]
	& 32	& 3.0e-11	& 2.6e-11	& 1.5e-13	& 1.9e-9	& 1.7e-9	& 3.0e-11	& 1.3e-7	& 1.2e-7	& 7.8e-9  \\[0.1cm] \hline
 \multicolumn{2}{|c|}{$\rho_3$}  & 6.1 & 6.1 & 7.8 & 6.3 & 6.3 & 7.9 & 6.9 & 6.9 & 8.1 \\[0.1cm] \hline
 
 \end{tabular}
\caption{Relative eigenvalue (EV) errors for $C^1$ quadratic and $C^2$ cubic isogeometric analysis with fully integrated mass ($G_{p+1}$), underintegrated mass ($R_p$), and dispersion-minimized mass (DMM) in 1D.}
\label{tab:eve1d} 
\end{table}

Fixing the wave numbers, Table \ref{tab:eve1d} shows their relative eigenvalue errors of the first, second, and fourth eigenmodes with respect to the mesh sizes. The convergence rates are denoted as $\rho_p$ for the $p$-th order approximation. For comparison purpose, Table \ref{tab:eve1d} also shows the errors while using the Radau rules. The errors for the Radau rules ($R_p$) are smaller than the errors of the fully integrated system ($G_{p+1}$). The dispersion-minimized mass has  two extra order of superconvergence. All these numerical results confirm our theoretical predictions detailed in the previous sections.

\begin{figure}[ht]
\centering
\includegraphics[height=5.8cm]{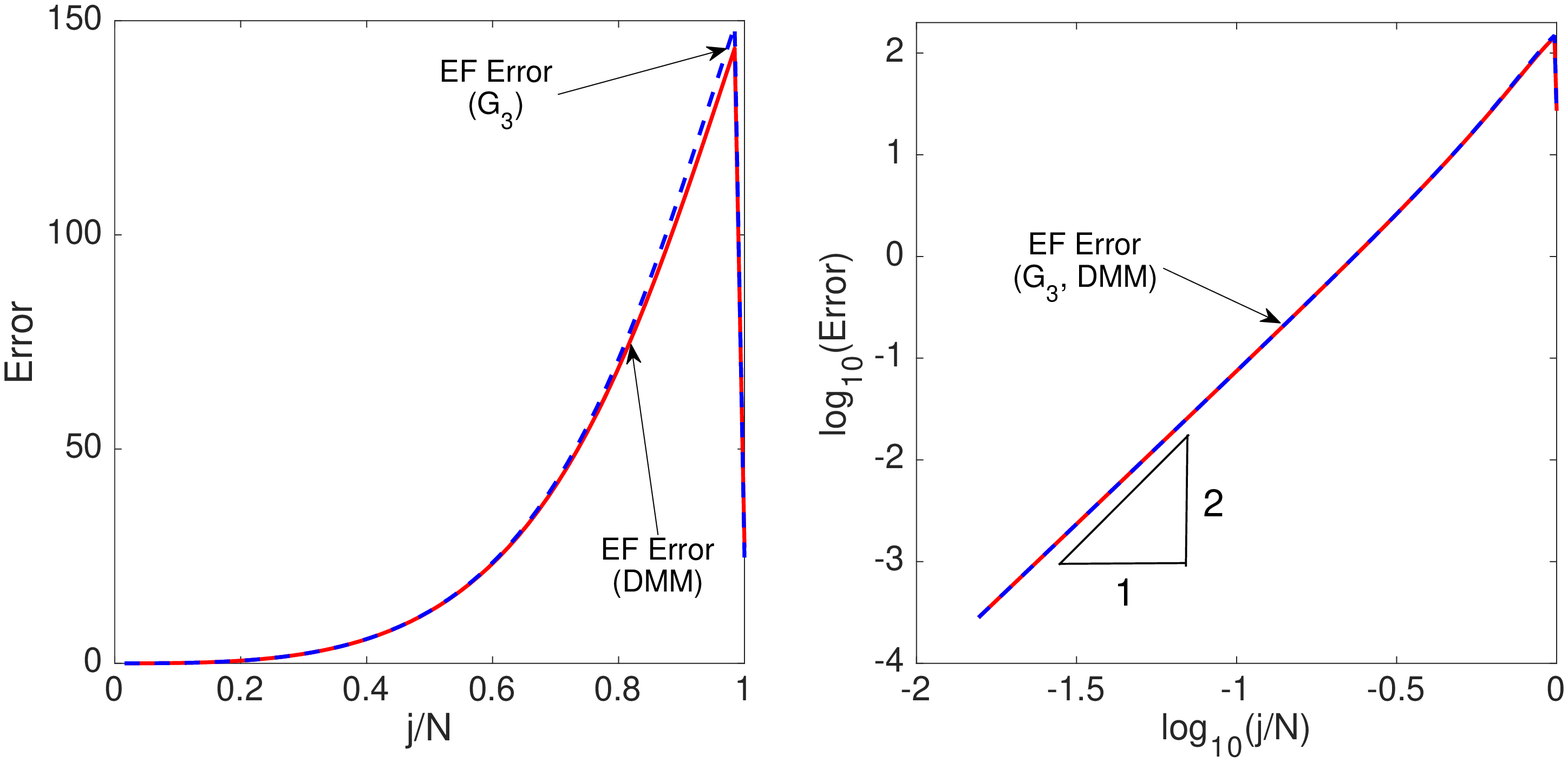} 
\caption{Eigenfunction (EF) error in energy norm for $C^1$ quadratic isogeometric analysis with fully integrated mass ($G_3$) and dispersion-minimized mass (DMM) in 1D.}
\label{fig:p2efe}
\includegraphics[height=5.8cm]{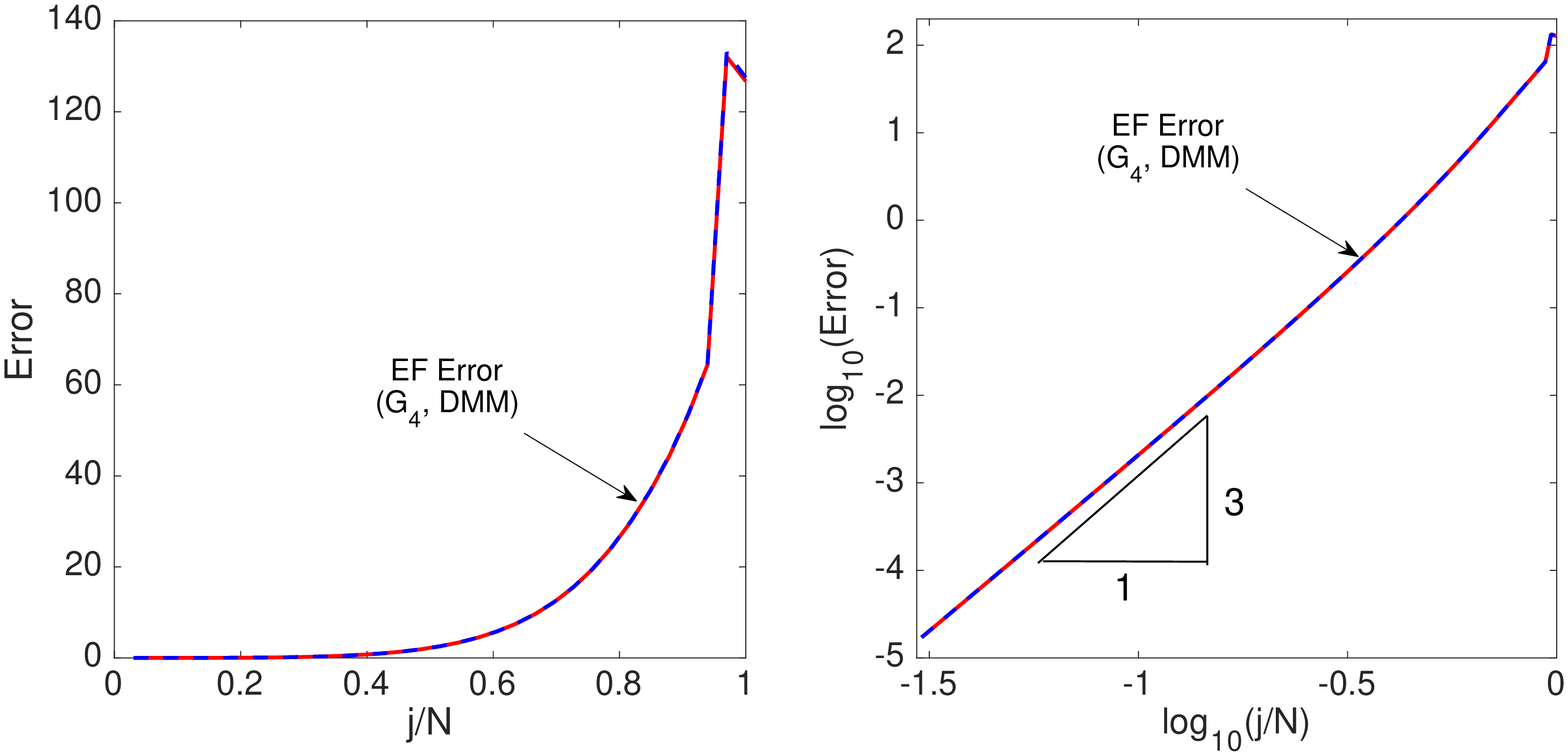} 
\caption{Eigenfunction (EF) error in energy norm for $C^2$ cubic isogeometric analysis with fully integrated mass ($G_4$) and dispersion-minimized mass (DMM) in 1D.}
\label{fig:p3efe}
\end{figure}

\subsection{Numerical study in 2D}
Let $\Omega = [0, 1] \times [0, 1]$.  The two dimensional differential eigenvalue problem \eqref{eq:pde} has exact eigenvalues and eigenfunctions 
\begin{equation}
\lambda_{jk} = ( j^2 + k^2 ) \pi^2, \quad \text{and} \quad u_{jk} = 2 \sin( j\pi x)\sin( k\pi y), \quad j,k = 1, 2, \cdots,
\end{equation}
respectively. We discretize the domain using a tensor-product structure. Figures \ref{fig:p2eve2d} and \ref{fig:p3eve2d} show the relative eigenvalue errors using isogeometric elements approximations for $p=2$ and $p=3$, respectively. The underlying meshes are of $32 \times 32$ uniform elements. We evaluate the isogeometric mass entries by full integration using Gauss rules and underintegration using Radau rules, as well as the dispersion-minimized mass. In general, the dispersion-minimized mass leads to the smallest relative eigenvalue errors while the Gauss rules results in the largest errors. The $p$-point Radau rule fully integrates polynomials up to order ($2p-2$), nevertheless, it behaves better in the simulation than that of the $p+1$ points Gauss rule which exactly integrates polynomials up to ($2p+1$) in 2D. This is also true for 1D and 3D. %, but it may not be true for non-uniform meshes. 

\begin{figure}[!t]
\centering
\includegraphics[height=7.0cm]{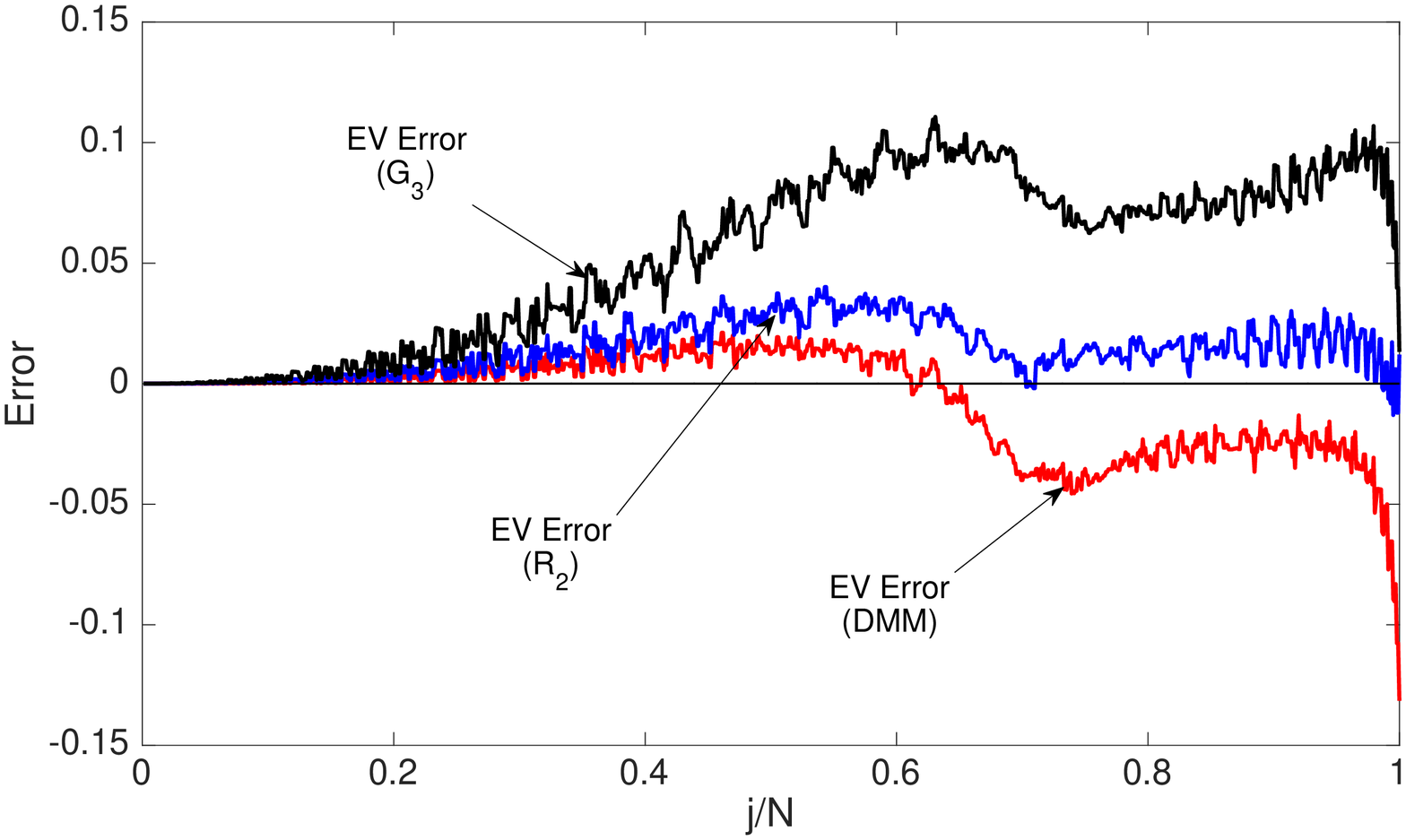} 
\caption{Relative eigenvalue (EV) errors for $C^1$ quadratic isogeometric analysis with fully integrated mass ($G_3$), underintegrated mass ($R_2$), and dispersion-minimized mass (DMM) in 2D. }
\label{fig:p2eve2d}
\end{figure}

\begin{figure}[!t]
\centering
\includegraphics[width=11cm,height=7.0cm]{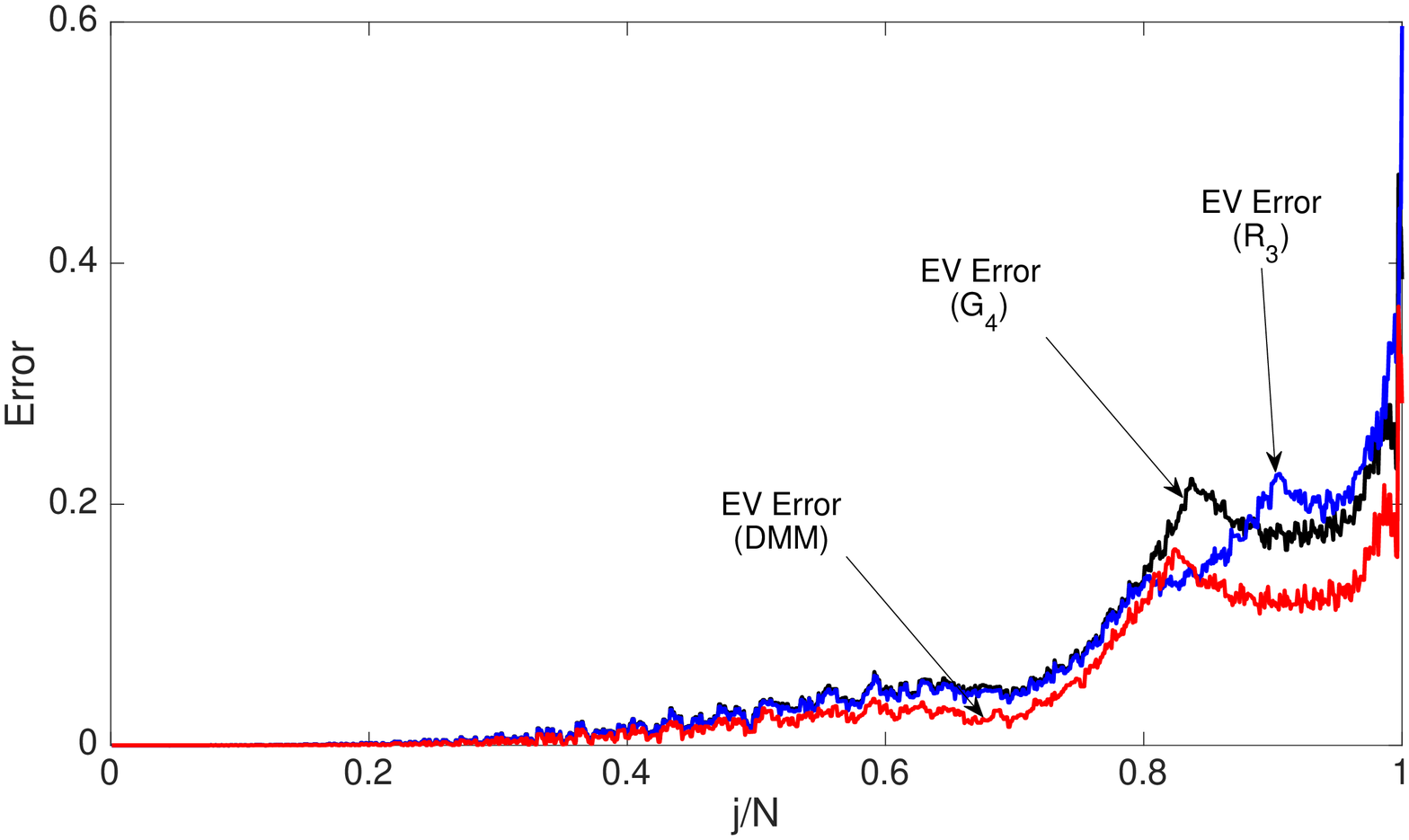} 
\caption{Relative eigenvalue (EV) errors for $C^2$ cubic isogeometric analysis with fully integrated mass ($G_4$), underintegrated mass ($R_3$), and dispersion-minimized mass (DMM) in 2D. }
\label{fig:p3eve2d}
\end{figure}

Table \ref{tab:eve2d} shows the relative eigenvalue errors while fixing the wave numbers and varying the mesh sizes. We present the errors for the first, second, and fourth eigenmodes. We observe the same convergence behavior as in 1D. In the view of Section \ref{sec:md}, the eigenvalue and eigenfunction error behaviors in multiple dimensions coincide with those in one dimension due to the tensor-product structure. Therefore, we omit the numerical results for three dimensions herein.

\begin{table}[ht]
\centering 
\begin{tabular}{| c | c | c | c | c | c | c | c| c| c| c| c|}
\hline
%\multicolumn{10}{|c|}{$p=2$}  \\[0.1cm] \hline
 \multicolumn{2}{|c|}{Set} &  \multicolumn{3}{c|} {$ |\lambda_1^h - \lambda_1| / \lambda_1$ }  & \multicolumn{3}{c|} {$ |\lambda_2^h - \lambda_2| /  \lambda_2$}  & \multicolumn{3}{c|} {$ |\lambda_4^h - \lambda_4| /  \lambda_4$}   \\[0.1cm] \hline
$p$ &  $N$ & $G_{p+1}$ & $R_p$ & DMM & $G_{p+1}$ & $R_p$ & DMM & $G_{p+1}$ & $R_p$ & DMM  \\[0.1cm] \hline

	& 8	& 3.4e-5	& 3.6e-6	& 6.7e-7	& 4.9e-4	& 6.7e-5	& 3.5e-5	& 6.0e-4	& 8.3e-5	& 4.3e-5 \\[0.1cm]
	& 16	& 2.1e-6	& 4.5e-7	& 1.0e-8	& 2.8e-5	& 6.3e-6	& 5.4e-7	& 3.4e-5	& 7.7e-6	& 6.7e-7 \\[0.1cm]
2	& 32	& 1.3e-7	& 3.5e-8	& 1.6e-10	& 1.7e-6	& 4.7e-7	& 8.4e-9	& 2.1e-6	& 5.8e-7	& 1.0e-8 \\[0.1cm]
	& 64	& 8.1e-9	& 2.4e-9	& 2.4e-12	& 1.1e-7	& 3.2e-8	& 1.3e-10	& 1.3e-7	& 3.9e-8	& 1.6e-10 \\[0.1cm] \hline
 \multicolumn{2}{|c|}{$\rho_2$}  & 4.0 & 3.5 & 6.0 & 4.1 & 3.7 & 6.0 & 4.1 & 3.7 & 6.0 \\[0.1cm] \hline

	& 4	& 9.7e-6	& 8.8e-6	& 1.7e-6	& 7.9e-4	& 7.4e-4	& 3.6e-4	& 9.9e-4	& 9.3e-4	& 4.5e-4 \\[0.1cm]
	& 8	& 1.3e-7	& 1.2e-7	& 7.3e-9	& 8.1e-6	& 7.3e-6	& 1.6e-6	& 1.0e-5	& 9.1e-6	& 2.0e-6 \\[0.1cm]
3	& 16	& 1.9e-9	& 1.7e-9	& 2.9e-11	& 1.0e-7	& 9.3e-8	& 6.1e-9	& 1.3e-7	& 1.2e-7	& 7.6e-9 \\[0.1cm]
	& 32	& 3.0e-11	& 2.6e-11	& 1.9e-13	& 1.6e-9	& 1.4e-9	& 2.4e-11	& 1.9e-9	& 1.7e-9	& 3.0e-11 \\[0.1cm] \hline
 \multicolumn{2}{|c|}{$\rho_3$}  & 6.1 & 6.1 & 7.7 & 6.3 & 6.3 & 8.0 & 6.3 & 6.3 & 8.0 \\[0.1cm] \hline
 
 \end{tabular}
\caption{Relative eigenvalue (EV) errors for $C^1$ quadratic and $C^2$ cubic isogeometric analysis with fully integrated mass ($G_{p+1}$), underintegrated mass ($R_p$), and dispersion-minimized mass (DMM) in 2D.}
\label{tab:eve2d} 
\end{table}

\section{Concluding remarks} \label{sec:conclusion}
The paper firstly establishes new facts on the stiffness and mass entries of the isogeometric elements using B-splines. These facts are essential to derive the dispersion and eigenvalue errors for arbitrary order B-splines. The natural and explicit relations between the stiffness and mass entries motivate us to develop the dispersion-minimized mass for the isogeometric elements. The dispersion-minimized mass leads to superconvergence of order $2p+2$ on eigenvalue errors. 

An equivalence between the dispersion-minimized mass and the optimal quadrature blending is then established in the view of the dispersion error. The optimally blended quadratures lead to the dispersion-minimized mass entries.  We generalize the optimal quadrature blending rules introduced in \cite{calo2017dispersion} from $p=7$ to arbitrary order. Comparing with the blending proposed in \cite{calo2017dispersion}, the dispersion minimizing quadrature is not limited to the blending  classical quadrature rules. For the $p$-th order isogeometric elements, the blending procedure can be applied to blend two arbitrary quadrature rules which fully integrate polynomials up to order $2p-2$. 

We generalize these results to mixed isogemetric elements for $2n$-order differential eigenvalue problems, which include the Cahn-Hilliard, Swift-Hohenberg, and Phase-field crystal operators. We will report our results in the near future.

Other future work includes (1) providing proofs for the identities and postulates we assert in this paper, (2) generalizations of the analysis for the dispersion-minimized mass (DMM) to non-uniform meshes and variable diffusion coefficients, (3) generalization of dispersion-minimized mass to isogeometric elements with variable continuities,  finite elements, and other methods. In general, the generalizations rely on the dispersion-minimized mass conditions \eqref{eq:dmmc0} posed for the mass entries, which is a set of linear problems on the mass entries. These are subject to future investigations.

\section*{Acknowledgments} 
This publication was made possible in part by the CSIRO Professorial Chair in Computational Geoscience at Curtin University and the Deep Earth Imaging Enterprise Future Science Platforms of the Commonwealth Scientific Industrial Research Organisation, CSIRO, of Australia. Additional support was provided by the European Union's Horizon 2020 Research and Innovation Program of the Marie Sk{\l}odowska-Curie grant agreement No. 644202 and the Curtin Institute for Computation. The J. Tinsley Oden Faculty Fellowship Research Program at the Institute for Computational Engineering and Sciences (ICES) of the University of Texas at Austin has partially supported the visits of VMC to ICES. The Spring 2016 Trimester on ``Numerical methods for PDEs", organised  with the collaboration of the Centre Emile Borel at the Institut Henri Poincare in Paris supported VMC's visit to IHP in October, 2016.

%\section*{References}

%\bibliographystyle{plain}
%\bibliographystyle{siam}
%\bibliography{ref.bib}
%\bibliographystyle{siamplain}

%\bibliographystyle{siam}

%\bibliography{igaref}

\end{document}